\documentclass[12pt,reqno]{amsart}

\usepackage{times}
\usepackage{graphicx}
\graphicspath{ {./images/} }

\usepackage{geometry}

\usepackage{amsmath,amsthm,amsfonts,amscd,amssymb}
\usepackage{pst-node}
\usepackage{pst-plot}
\usepackage[all]{xy}
\usepackage{stackengine}
\usepackage{multirow}

\usepackage{tikz, braids}
\usetikzlibrary{decorations.pathreplacing, arrows}

\newtheorem{theorem}{Theorem}[section]
\newtheorem{corollary}[theorem]{Corollary}
\newtheorem{definition}[theorem]{Definition}
\newtheorem{lemma}[theorem]{Lemma}
\newtheorem{remark}[theorem]{Remark}
\newtheorem{proposition}[theorem]{Proposition}
\newtheorem{example}[theorem]{Example}

\numberwithin{theorem}{section}
\numberwithin{equation}{section}

\newcommand{\norm}[1]{\left\Vert#1\right\Vert}

\newcommand{\la}{\langle}
\newcommand{\ra}{\rangle}

\newcommand{\Comp}{\mathbb{C}}

\newcommand{\g}{\mathbb{G}}

\newcommand{\n}{\mathbb{N}}
\newcommand{\re}{\mathbb{R}}

\begin{document}

\title[Twisted Fourier transforms on non-Kac compact quantum groups]{Twisted Fourier transforms on non-Kac compact quantum groups}


\maketitle

\begin{abstract}
We introduce an analytic family of twisted Fourier transforms $\left\{\mathcal{F}^{(x)}_p\right\}_{x\in \re,p\in [1,2)}$ for non-Kac compact quantum groups and establish a sharpened form of the Hausdorff-Young inequality in the range $0\leq x \leq 1$.  Furthermore, we prove that the range $0\leq x \leq 1$ is both necessary and sufficient for the boundedness of $\mathcal{F}^{(x)}_p$ under the assumption of sub-exponential growth on the dual discrete quantum group. We also show that the range of boundedness of $\mathcal{F}^{(x)}_p$ can be strictly extended beyond $[0,1]$ for certain non-Kac and non-coamenable free orthogonal quantum groups. As applications, we derive a stronger form of the twisted rapid decay property for polynomially growing non-Kac discrete quantum groups, including the duals of the Drinfeld-Jimbo $q$-deformations, and construct an explicit contractive, but non-completely bounded, representation of the convolution algebra of any non-Kac free orthogonal quantum group.
\end{abstract}

\section{Introduction}

Abstract harmonic analysis on nonabelian compact groups has a rich history, with a well-established generalization of the Fourier transform \cite{HeRo63,HeRo70,Rud90,Fol95,ReSt00}. Although several formulations exist, they are essentially equivalent from the viewpoint of harmonic analysis. This theory has since been extended into the broader framework of {\it noncommutative harmonic analysis}, which encompasses various models of quantum spaces. Notable examples include quantum tori \cite{Ric16,XXX16,XXY18}, quantum Euclidean spaces \cite{MSX20,GAJP21,JMPX21,HoLaWa23}, group von Neumann algebras \cite{JMP14,JMP18,PRd22,CCJP22}, and locally compact quantum groups \cite{LWW17,CaVo22,HWW24}.

While substantial progress has been made in settings where the associated noncommutative measure is {\it tracial}, we will focus on {\it non-Kac} compact quantum groups whose canonical noncommutative measure is {\it non-tracial}. This gives rise to fundamentally different analytic behavior compared to the tracial situation. For instance, recent studies have demonstrated that certain analytic results on non-Kac compact quantum groups are in sharp contrast with their classical counterparts \cite{Wa17,BVY21,You22,You24}. 

Motivated by such contrasts, this paper investigates foundational aspects of harmonic analysis, such as the Fourier transform and the Hausdorff–Young inequality, to understand the fundamental structural origins of such unexpected phenomena. Let us denote by $\mathcal{F}^{(0)}_p(f)=\widehat{f}$ the standard Fourier transform on a compact quantum group. 

In Section \ref{sec-twisted}, we introduce an analytic family of {\it twisted Fourier transforms} $\left\{\mathcal{F}^{(x)}_p\right\}_{x\in \re, p\in [1,\infty]}$ defined by
\begin{equation}
    \mathcal{F}^{(x)}_p(f)=Q^{-(\frac{1}{p}-\frac{1}{2})x}\widehat{f}Q^{(\frac{1}{p}-\frac{1}{2})x}
\end{equation}
where the modular operator $Q=(Q_{\alpha})_{\alpha\in \text{Irr}(\g)}$ is defined in Subsection \ref{sec-QG}. A central question addressed in this paper is whether the twisted Fourier transform $\mathcal{F}^{(x)}_p$ extends to a bounded map from $L^p(\g)$ to $\ell^{p'}(\widehat{\g})$ where $\widehat{\g}$ denotes the dual discrete quantum group of $\g$. The main result of this section is the establishment of the following {\it strong Hausdorff-Young inequality} (Theorem \ref{thm-main}) for all $1\leq p<2$:
\begin{align}
\label{eq01}&\sup_{0\leq x \leq 1} \norm{Q^{-(\frac{1}{p}-\frac{1}{2})x}\widehat{f}Q^{(\frac{1}{p}-\frac{1}{2})x}}_{\ell^{p'}(\widehat{\g})}\leq \norm{f}_{L^p(\g)}.
\end{align}
In contrast, for $p>2$, we show that the twisted Fourier transform $\mathcal{F}^{(x)}_p$ extends to a bounded map from $L^p(\g)$ to $\ell^{p'}(\widehat{\g})$ only when $\g$ is finite-dimensional (Proposition \ref{prop51}), similarly to the classical situation.

 In Section \ref{sec-RD}, we establish the following dual formulation of the strong Hausdorff-Young inequality (Proposition \ref{prop-duality}) for $1\leq p\leq 2$:
 \begin{equation}
\label{eq04}\norm{g}_{L^{p'}(\g)}\leq \inf_{ 0 \leq x \leq 1}\norm{Q^{(\frac{1}{p}-\frac{1}{2})x}\widehat{g}Q^{-(\frac{1}{p}-\frac{1}{2})x}}_{\ell^p(\widehat{\g})}.     
 \end{equation}
 In particular, we demonstrate that the inequality \eqref{eq04} for the case $p=1$ is a strengthened form of the so-called {\it twisted rapid decay property} under the assumption of polynomial growth on $\widehat{\g}$ (Corollary \ref{cor-strong-twisted-RD}). For example, we establish this strengthened form of the twisted rapid decay property for the duals of Drinfeld-Jimbo $q$-deformations.

In Section \ref{sec-optimality}, we investigate whether the range $0\leq x \leq 1$ is optimal for inequalities \eqref{eq01} and \eqref{eq04}. More precisely, we define
\begin{align}
    I(\g,p)&=\left\{x\in \re ~\Bigr|~ \mathcal{F}^{(x)}_p:L^p(\g)\rightarrow \ell^{p'}(\widehat{\g})\text{ is bounded}\right\}.
\end{align}
Note that \eqref{eq01} guarantees that $[0,1]\subseteq I(\g,p)$ for any compact quantum group $\g$ and $1\leq p<2$. Subsection \ref{subsec-coamenable} focuses on non-Kac and coamenable compact quantum groups, and shows that equality $[0,1]=I(\g,p)$ holds for all $1\leq p<2$ under the assumption of sub-exponential growth on $\widehat{\g}$ (Theorem \ref{thm50}). This class includes the Drinfeld-Jimbo $q$-deformations $G_q$. Subsection \ref{subsec-non-coamenable} focuses on non-Kac and non-coamenable {\it free orthogonal quantum groups} $O_F^+$. Using the recently established Haagerup inequality on $O_F^+$ \cite{BVY21,You22}, we show that the inclusion $[0,1]\subsetneq I(O_F^+,p)$ is strict for certain non-Kac and non-coamenable free orthogonal quantum groups $O_F^+$ satisfying a technical assumption $\norm{F}_{op}^6<\frac{d_1+\sqrt{d_1^2-4}}{2}$ where $d_1=\text{Tr}(F^*F)$ (Corollary \ref{cor51}). A summary of the main conclusions regarding $I(\g,p)$ is presented in the following table:

\small
\begin{table}[h]
\centering
\begin{tabular}{|c|c|c|c|}
\hline
 & Kac & non-Kac, coamenable & non-Kac, non-coamenable \\
\hline
$1\leq p<2$ & $I(\g,p)=\mathbb{R}$  & $I(G_q,p)=[0,1]$ & $I(O_F^+,p)\supsetneq [0,1]$\\
& (Trivial) & (Corollary \ref{cor50.5}) & (Corollary \ref{cor51})\\
\hline
$2<p\leq \infty$ & \multicolumn{3}{c|}{$I(\g,p)=\left\{\begin{array}{ll}\mathbb{R}&,~\g\text{ is finite-dimensional}\\ \emptyset&,~\g\text{ is infinite-dimensional} \end{array} \right .$ (Proposition \ref{prop51})} \\
\hline
\end{tabular}
\end{table}
\normalsize

In Section 6, as an application, we introduce another family of twisted Fourier transforms $\pi^{(x)}$, which contributes to the study of the {\it similarity problem} for locally compact quantum groups. The theory associated with $\pi^{(x)}$ runs in parallel with that of $\mathcal{F}_1^{(x)}$; however, a notable advantage of this approach is that each $\pi^{(x)}$ defines a contractive algebra homomorphism from the convolution algebra $L^1(\g)$ into $\ell^{\infty}(\widehat{\g})$ (Proposition \ref{prop500}). Moreover, we prove that the contractive twisted Fourier transform $\pi^{(1)}$ fails to be completely bounded for any non-Kac free orthogonal quantum groups $O_F^+$ (Theorem \ref{thm500}).

\section{Preliminaries}	

\subsection{Compact quantum groups}\label{sec-QG}

Within the von Neumann algebraic framework, in the sense of \cite{KuVa00,KuVa03}, a {\it compact quantum group} $\g$ is given by a triple $(L^{\infty}(\g),\Delta,h)$ where
\begin{enumerate}
    \item $L^{\infty}(\g)$ is a von Neumann algebra,
    \item $\Delta:L^{\infty}(\g)\rightarrow L^{\infty}(\g)\overline{\otimes }L^{\infty}(\g)$ is a normal unital $*$-homomorphism satisfying
    \begin{equation}
        (\text{id}\otimes \Delta)\Delta= (\Delta \otimes\text{id})\Delta,
    \end{equation}
    \item $h$ is a normal faithful state on $L^{\infty}(\g)$ satisfying
    \begin{equation}
        (\text{id}\otimes h)\Delta = h(\cdot) 1_{L^{\infty}(\g)}=(h\otimes \text{id})\Delta.
    \end{equation}
    We call $h$ the {\it Haar state} on $\g$.
\end{enumerate}

A {\it finite dimensional unitary representation} of $\g$ is given by a unitary $u=\sum_{i,j=1}^{n_u}e_{ij}\otimes u_{ij}\in M_{n_u}(\Comp)\otimes L^{\infty}(\g)$ satisfying
\begin{equation}
    \Delta(u_{ij})=\sum_{k=1}^{n_u} u_{ik}\otimes u_{kj}
\end{equation}
for all $1\leq i,j\leq n_u$. In addition, $u$ is called {\it irreducible} if 
\begin{equation}
    (T\otimes 1)u= u(T\otimes 1) \Rightarrow T\in \Comp\cdot \text{Id}_{n_u}.
\end{equation}
Finite dimensional unitary representations $u,v\in M_{n}(\Comp)\otimes L^{\infty}(\g)$ are called {\it unitarily equivalent} if there exists an $n\times n$ unitary matrix $W$ such that
\begin{equation}
    (W\otimes 1)u=v(W\otimes 1).
\end{equation}
We denote by $\text{Irr}(\g)$ the set of all irreducible (finite dimensional) unitary representations of $\g$ up to the unitary equivalence.

For each $\alpha\in \text{Irr}(\g)$, let us denote by $u^{\alpha}\in M_{n_{\alpha}}(\Comp)\otimes L^{\infty}(\g)$ a representative irreducible unitary representation. The {\it contragredient representation} $(u^{\alpha})^c=\sum_{i,j=1}^{n_u}e_{ij}\otimes (u^{\alpha}_{ij})^*\in M_{n_u}(\Comp)\otimes L^{\infty}(\g)$ is not a unitary, but there exists a unique positive invertible matrix $Q_{\alpha}\in M_{n_{\alpha}}(\Comp)$ such that
\begin{equation}
    (Q_{\alpha}^{\frac{1}{2}}\otimes 1)(u^{\alpha})^c (Q_{\alpha}^{-\frac{1}{2}}\otimes 1)
\end{equation}
is a unitary and $\text{Tr}(Q_{\alpha})=\text{Tr}(Q_{\alpha}^{-1})$. We denote by
\begin{equation}
    d_{\alpha}=\text{Tr}(Q_{\alpha})=\text{Tr}(Q_{\alpha}^{-1}),
\end{equation}
and call it the {\it quantum dimension} of $u^{\alpha}$. We may assume that the matrices $Q_{\alpha}$ are diagonal by taking suitable orthonormal bases of $\Comp^{n_{\alpha}}$. If $Q_{\alpha}=\text{Id}_{n_{\alpha}}$ for all $\alpha\in \text{Irr}(\g)$, then $\g$ is called {\it Kac type}. This is equivalent to that the Haar state $h$ is {\it tracial}, i.e.
\begin{equation}
    h(ab)=h(ba),~a,b\in L^{\infty}(\g).
\end{equation}

The space of {\it polynomials} is defined as
\begin{equation}
    \text{Pol}(\g)=\text{span}\left\{u^{\alpha}_{ij}: \alpha\in \text{Irr}(\g),~1\leq i,j\leq n_{\alpha}\right\},
\end{equation}
and we call $u^{\alpha}_{ij}$ a {\it matrix element}. For matrix elements $u^{\alpha}_{ij}$ and $u^{\beta}_{kl}$, the {\it Schur orthogonality relation} states that
\begin{align}
    h\left ( (u^{\beta}_{kl})^*u^{\alpha}_{ij}\right )&=\frac{\delta_{\alpha\beta}\delta_{ik}\delta_{jl}(Q_{\alpha})_{ii}^{-1}}{d_{\alpha}} \\
    h\left ( u^{\beta}_{kl} (u^{\alpha}_{ij})^*\right )&=\frac{\delta_{\alpha\beta}\delta_{ik}\delta_{jl}(Q_{\alpha})_{jj}}{d_{\alpha}}.
\end{align}


 A compact quantum group $\g$ is called a {\it compact matrix quantum group} if there exists a unitary representation $u\in M_n(\Comp)\otimes L^{\infty}(\g)$ satisfying that
\begin{equation}
   \text{Pol}(\g)=\text{span}\left\{ u_{i_1j_1}u_{i_2j_2}\cdots u_{i_kj_k}: k\in \n,~1\leq i,j\leq n \right\} .
\end{equation}
If we write $V_0=\Comp\cdot 1_{L^{\infty}(\g)}$ and $V_k=\text{span}\left\{ u_{i_1j_1}u_{i_2j_2}\cdots u_{i_kj_k}: ~1\leq i_s,j_s\leq n \right\}$ for all $k\in \n$, then a natural length function on $\text{Irr}(\g)$ is defined by
\begin{equation}
  |\alpha|= \inf\left\{ k\in \n_0:  \left\{u^{\alpha}_{ij}\right\}_{1\leq i,j\leq n_{\alpha}}\subseteq V_k \right\}.
\end{equation}

Throughout this paper, we will use the following lemma frequently.

\begin{lemma}\label{lem20} \cite[Section 6]{KrSo18}
Let $\g$ be a compact matrix quantum group of non-Kac type and let $|\cdot |$ be the natural length function on $\text{Irr}(\g)$. Then there exists a sequence $(\alpha_k)_{k\in \n}\subseteq \text{Irr}(\g)$ such that
\begin{itemize}
    \item $|\alpha_{k+1}|\leq 2|\alpha_k|$ for all $k\in \n$,
\item $\norm{Q_{\alpha_1}}>1$,
    \item $\norm{Q_{\alpha_{k+1}}}=\norm{Q_{\alpha_{k}}}^{2}$ for all $k\in \n$. 
\end{itemize}
\end{lemma}

For more details of compact quantum groups, we refer the readers to \cite{Wo87a,Wo87b,Ti08,NeTu13}.

\subsection{Noncommutative $L^p$-spaces}\label{sec-Lp}

Let us denote by $L^1(\g)$ the predual of the von Neumann algebra $L^{\infty}(\g)$. Since $h$ is faithful, we have a natural embedding $\iota:L^{\infty}(\g)\hookrightarrow L^1(\g)$ defined by $\iota(a)=h(\cdot~ a)$, i.e.
\begin{equation}
    [\iota(a)](x)=h(xa),~x\in L^{\infty}(\g).
\end{equation}
This allows us to understand $(L^{\infty}(\g),L^1(\g))$ as a compatible pair of Banach spaces, so we can define the {\it noncommutative $L^p$-space} $(1\leq p\leq \infty)$ as the complex interpolation space 
\begin{equation}
    L^p(\g)= (L^{\infty}(\g),L^1(\g))_{\theta}
\end{equation}
for $\theta=\frac{1}{p}$. For any $1\leq p<\infty$, the space $\text{Pol}(\g)$ is dense in $L^p(\g)$.

From the {\it Tomita-Takesaki theory}, there exists a {\it modular automorphism group} $(\sigma_z)_{z\in \Comp}$ on $\text{Pol}(\g)$ satisfying
\begin{equation}
    h(ab)=h(b\sigma_{-i}(a)),~a,b\in \text{Pol}(\g).
\end{equation}
Specifically, for any $z\in \Comp$, the modular automorphism $\sigma_z$ is determined by
\begin{equation}
    \sigma_z\left (u^{\alpha}_{kl} \right )=(Q_{\alpha})_{kk}^{iz}(Q_{\alpha})_{ll}^{iz}u^{\alpha}_{kl}
\end{equation}
for all $\alpha\in \text{Irr}(\g)$ and $1\leq k,l\leq n_{\alpha}$.

For any $1\leq p\leq \infty$, let us consider the conjugate $p'$ such that $\frac{1}{p}+\frac{1}{p'}=1$. Then a natural dual pairing between $L^{p'}(\g)$ and $L^p(\g)$ is given by
\begin{equation}
    \la f,g\ra_{L^{p'}(\g),L^p(\g)}=h\left (\sigma_{\frac{i}{p'}}(f)g\right ).
\end{equation}
Combining this with \cite[Lemma 3.4. (c)]{Wa17}, we obtain 
\begin{equation}
 \label{eq21} \norm{f}_{L^{p'}(\g)}=\norm{\sigma_{-\frac{i}{p'}}(f^*)}_{L^{p'}(\g)}=\sup_{g\in \text{Pol}(\g): \norm{g}_{L^{p}(\g)}\leq 1}h(f^*g).
    \end{equation}
for all $f\in \text{Pol}(\g)$ with $1\leq p\leq \infty$.

When we understand $(M_n(\Comp),\text{Tr})$ as a noncommutative measure space, we denote by $S^p_n$ the matrix algebra $M_n(\Comp)$ with the norm structure
\begin{equation}
 \norm{A}_{S^p_n}=\left\{\begin{array}{cl} 
 \text{Tr}(|A|^p)^{\frac{1}{p}}&,~1\leq p<\infty,\\
\norm{A}_{op} &,~p=\infty. \end{array} \right .
\end{equation}
Let us denote by $S=\left\{z\in \Comp: 0<\text{Re}(z)<1\right\}$ and by $\mathcal{C}(A)$ the set of all continuous bounded functions $f:\overline{S}\rightarrow M_n(\Comp)$ such that $f|_S:S\rightarrow M_n(\Comp)$ is analytic and $f(\frac{1}{p})=A$. The finite dimensional {\it Schatten $p$-space} $S^p_n$ is the complex interpolation space $(M_{n},S^1_n)_{\theta}$ with $\theta=\frac{1}{p}$ in the sense that
\begin{equation}
    \norm{A}_{S^p_n}=\inf_{f\in \mathcal{C}(A)}\left\{\sup_{s\in \re}\norm{f(is)}_{M_n},\sup_{s\in \re}\norm{f(1+is)}_{S^1_n}\right\}.
\end{equation}
Let us consider the following weighted norm structure
\begin{equation}
    \norm{A}_{X^1_n}=\text{Tr}(|G^x A G^{1-x}|),~A\in M_n(\Comp)
\end{equation}
with an invertible positive $n\times n$ matrix $G$. Then, by definition, the norm structure of the complex interpolation space $X^p_n=(M_n, X^1_n)_{\frac{1}{p}}$ is given by
\begin{align}
    \norm{A}_{X^p_n}=\inf_{g\in \mathcal{C}(A)} \left\{\sup_{s\in \re}\norm{g(is)}_{M_n},\sup_{s\in \re}\norm{g(1+is)}_{X^1_n}\right\}.
\end{align}
For any $g\in \mathcal{C}(A)$, let us consider $f(z)=G^{zx}g(z)G^{z(1-x)}$. This defines an one-to-one correspondence between $\mathcal{C}(A)$ and $\mathcal{C}(G^{\frac{x}{p}}AG^{\frac{1-x}{p}})$ and we have
\begin{align}
    &\norm{f(is)}_{M_n}=\norm{g(is)}_{M_n},\\
    &\norm{f(1+is)}_{S^1_n}=\norm{G^xg(1+is)G^{1-x}}_{S^1_n}=\norm{g(1+is)}_{X^1_n}.
\end{align}
Thus, we can see that 
\begin{align}
    \norm{A}_{X^p_n}&=\inf_{f\in \mathcal{C}(G^{\frac{x}{p}}AG^{\frac{1-x}{p}})} \left\{\sup_{s\in \re}\norm{f(is)}_{M_n},\sup_{s\in \re}\norm{f(1+is)}_{S^1_n}\right\}\\
 \label{eq20.7}   &=\norm{G^{\frac{x}{p}}AG^{\frac{1-x}{p}}}_{S^p_n}.
\end{align}
For more general discussions, see Section 7 and Theorem 11.1 in \cite{Kos84}.

\subsection{Discrete quantum group and Fourier transform}\label{sec-Fourier}

In view of the celebrated {\it Pontryagin duality}, each compact quantum group $\g$ admits a {\it (dual) discrete quantum group} $\widehat{\g}$ whose associated function space is given by a direct sum of matrix algebras
\begin{equation}
    \ell^{\infty}(\widehat{\g})=\ell^{\infty}-\bigoplus_{\alpha\in \text{Irr}(\g)}M_{n_{\alpha}}(\Comp).
\end{equation}
For $p=\infty$, we simply write $M_n=S^{\infty}_n$.

For any $A=(A(\alpha))_{\alpha\in \text{Irr}(\g)}\in \ell^{\infty}(\widehat{\g})$, we denote the {\it support} of $A$ by
\begin{equation}
    \text{supp}(A)=\left\{\alpha\in \text{Irr}(\g): A(\alpha)\neq 0 \right\},
\end{equation}
and define the following subspace
\begin{equation}
    c_{00}(\widehat{\g})=\left\{A\in \ell^{\infty}(\widehat{\g}): \text{supp}(A)\text{ is finite}\right\}.
\end{equation}

The {\it dual Haar weight} $\widehat{h}:\ell^{\infty}(\widehat{\g})_+\rightarrow [0,\infty]$ is given by
\begin{equation}
    \widehat{h}(A)=\sum_{\alpha\in \text{Irr}(\g)}d_{\alpha}\text{Tr}(A(\alpha)Q_{\alpha})
\end{equation}
for all $A=(A(\alpha) )_{\alpha\in \text{Irr}(\g)}\in \ell^{\infty}(\widehat{\g})_+$. 

Let us denote by $Q=(Q_{\alpha})_{\alpha\in \text{Irr}(\g)}\in \displaystyle \prod_{\alpha\in \text{Irr}(\g)}M_{n_{\alpha}}(\Comp)$. The non-commutative $\ell^p$-space ($1\leq p<\infty$) on $\widehat{\g}$ is explicitly defined as
\begin{equation}
    \ell^p(\widehat{\g})= \left\{A\in \ell^{\infty}(\widehat{\g}): \widehat{h}\left ( \left |AQ^{\frac{1}{p}} \right |^p\right )<\infty \right \}
\end{equation}
and the $\ell^p$-norm of $A=(A(\alpha) )_{\alpha\in \text{Irr}(\g)}\in \ell^{p}(\widehat{\g})$ is explicitly given by
\begin{align}
    \norm{A}_{\ell^p(\widehat{\g})}= \left \{\begin{array}{lll} \left ( \sum_{\alpha\in \text{Irr}(\g)} d_{\alpha} \norm{A(\alpha)Q_{\alpha}^{\frac{1}{p}} }_{S^p_{n_{\alpha}}}^p \right )^{\frac{1}{p}}&,~1\leq p<\infty, \\
    \sup_{\alpha\in \text{Irr}(\g)}\norm{A(\alpha)}_{op}&,~p=\infty.
    \end{array} \right .
\end{align}

Similarly as in Subsection \ref{sec-Lp}, a natural dual pairing between $\ell^p(\widehat{\g})$ and $\ell^{p'}(\widehat{\g})$ is given by
\begin{align}
&\la B,A\ra_{\ell^{p'}(\widehat{\g}), \ell^{p}(\widehat{\g})}=\widehat{h}\left ( \sigma_{\frac{i}{p'}}(B)  A \right ),
\end{align}
and we have
\begin{align}
  \label{eq22}  &\norm{B}_{\ell^{p'}(\g)}=\norm{\sigma_{-\frac{i}{p'}}(B^*)}_{\ell^{p'}(\widehat{\g})}=\sup_{A\in c_{00}(\widehat{\g}): \norm{A}_{\ell^p(\widehat{\g})}\leq 1}\widehat{h}(B^*A).
\end{align}

For any $\varphi\in L^1(\g)$, the {\it Fourier coefficient} at $\alpha\in \text{Irr}(\g)$ is defined by
\begin{equation}\label{Fourier1}
  \widehat{\varphi}(\alpha)=  (\text{id}\otimes \varphi)((u^{\alpha})^*),
\end{equation}
and call $\widehat{\varphi}=(\widehat{\varphi}(\alpha))_{\alpha\in \text{Irr}(\g)}$ the {\it Fourier transform }of $\varphi$. We call 
\begin{equation}
    \sum_{\alpha\in \text{Irr}(\g)}\sum_{i,j=1}^{n_{\alpha}}d_{\alpha}(\widehat{\varphi}(\alpha)Q_{\alpha})_{ij}u^{\alpha}_{ji}
\end{equation}
the {\it Fourier series} of $\varphi\in L^1(\g)$, and write
\begin{equation}
    \varphi \sim \sum_{\alpha\in \text{Irr}(\g)}\sum_{i,j=1}^{n_{\alpha}}d_{\alpha}(\widehat{\varphi}(\alpha)Q_{\alpha})_{ij}u^{\alpha}_{ji}.
\end{equation}
In particular, for any $f\in \text{Pol}(\g)\subseteq L^1(\g)$, we have 
\begin{equation}
    f= \sum_{\alpha\in \text{Irr}(\g)}\sum_{i,j=1}^{n_{\alpha}}d_{\alpha}(\widehat{f}(\alpha)Q_{\alpha})_{ij}u^{\alpha}_{ji}.
\end{equation}




\section{Twisted Fourier transforms}\label{sec-twisted}

In this section, let $\g$ be a general compact quantum group. The main goal is to study the linear maps $\mathcal{F}^{(x)}_p$ defined below.

\begin{definition}
    For any $x\in \re$ and $p\in [1,\infty]$, let us consider a linear map $\mathcal{F}^{(x)}_p:\text{Pol}(\g)\rightarrow c_{00}(\widehat{\g})$ given by
    \begin{equation}
        \mathcal{F}^{(x)}_p(f)= Q^{-(\frac{1}{p}-\frac{1}{2})x}\widehat{f}Q^{(\frac{1}{p}-\frac{1}{2})x},~f\in \text{Pol}(\g).
    \end{equation}
    We call these maps $\mathcal{F}^{(x)}_p$ {\it twisted Fourier transforms} on $\g$. In particular, we call $\mathcal{F}^{(0)}_p:f\mapsto \widehat{f}$ the {\it standard Fourier transform}.
\end{definition}

In this paper, we will focus on the question of whether the twisted Fourier transform $\mathcal{F}^{(x)}_p$ extends to a bounded map
\begin{equation}
    \mathcal{F}^{(x)}_p:L^p(\g)\rightarrow \ell^{p'}(\widehat{\g}).
\end{equation}
In other words, we study the Hausdorff-Young inequalities up to constants for the twisted Fourier transforms. Let us define
\begin{equation}
    I(\g,p)=\left\{x\in \re: \mathcal{F}^{(x)}_p:L^p(\g)\rightarrow \ell^{p'}(\widehat{\g})\text{ is bounded} \right\}.
\end{equation}

\subsection{Unboundedness for the cases $p>2$}\label{subsec-3-unbdd}

As in the classical situation, for the cases $p>2$, let us demonstrate that there is no universal constant $K=K(\g,p,x)$ such that
\begin{equation}
    \norm{\mathcal{F}^{(x)}_p(f)}_{\ell^{p'}(\widehat{\g})}\leq K\norm{f}_{L^p(\g)},~f\in L^p(\g),
\end{equation}
i.e. $I(\g,p)=\emptyset$ if $\g$ is infinite-dimensional. To prove this, let us begin with the following Lemma, which will be used frequently throughout this paper.

\begin{lemma}\label{lem30}
Let $\tau_{\alpha}$ be a bijective function on $\left\{1,2,\cdots,n_{\alpha}\right\}$ for each $\alpha\in \text{Irr}(\g)$, and consider 
\begin{equation}
    f=\sum_{\alpha\in \text{Irr}(\g)}\sum_{i=1}^{n_{\alpha}} y^{\alpha}_i u^{\alpha}_{i \tau_{\alpha}(i)}\in \text{Pol}(\g).
\end{equation}
Then, for $p=1$, we have
\begin{equation}
    \norm{\mathcal{F}^{(x)}_1(f)}_{\ell^{\infty}(\widehat{\g})}=\sup_{\alpha\in \text{Irr}(\g)}\sup_{1\leq i\leq n_{\alpha}} \left [ (Q_{\alpha})_{\tau_{\alpha}(i)\tau_{\alpha}(i)}^{-1}(Q_{\alpha})_{ii} \right ]^{\frac{x}{2}}\cdot \frac{(Q_{\alpha})_{ii}^{-1}}{d_{\alpha}}|y^{\alpha}_i|
\end{equation}
and, for any $1<p\leq \infty$, we have
\begin{align}
   &\norm{\mathcal{F}^{(x)}_p(f)}_{\ell^{p'}(\widehat{\g})}\\
   &= \left ( \sum_{\alpha\in \text{Irr}(\g)} \sum_{i=1}^{n_{\alpha}}  \left [ (Q_{\alpha})_{\tau_{\alpha}(i)\tau_{\alpha}(i)}^{-1} (Q_{\alpha})_{ii} \right ]^{(\frac{p'}{2}-1)x} \left [\frac{(Q_{\alpha})_{ii}^{-1}}{ d_{\alpha} } \right ]^{p'-1} \left |y^{\alpha}_i \right |^{p'} \right )^{\frac{1}{p'}}.
\end{align}
In particular, we have
\begin{align}
    &\norm{\mathcal{F}^{(x)}_p(u^{\alpha}_{ij})}_{\ell^{p'}(\widehat{\g})}= [(Q_{\alpha})_{jj}^{-1}(Q_{\alpha})_{ii}]^{(\frac{1}{p}-\frac{1}{2})x}\cdot \left ( \frac{(Q_{\alpha})_{ii}^{-1}}{d_{\alpha}} \right )^{\frac{1}{p}}
\end{align}
for any matrix element $u^{\alpha}_{ij}$ and $1\leq p\leq \infty$.
\end{lemma}

\begin{proof}

Note that the Fourier transform of $f= \displaystyle \sum_{\alpha\in \text{Irr}(\g)}\sum_{i=1}^{n_{\alpha}} y^{\alpha}_i u^{\alpha}_{i \tau_{\alpha}(i)}\in \text{Pol}(\g)$ is given by
\begin{equation}
    \widehat{f}=\bigoplus_{\alpha\in \text{Irr}(\g)}\sum_{i=1}^{n_{\alpha}} \frac{y^{\alpha}_i(Q_{\alpha})_{ii}^{-1}}{d_{\alpha}}E^{\alpha}_{\tau_{\alpha}(i)i},
\end{equation}
so the twisted Fourier transform is given by
\begin{align}
  &\mathcal{F}^{(x)}_p(f)=Q^{-(\frac{1}{p}-\frac{1}{2})x}\widehat{f}Q^{(\frac{1}{p}-\frac{1}{2})x}\\
  &= \bigoplus_{\alpha\in \text{Irr}(\g)} \sum_{i=1}^{n_{\alpha}} (Q_{\alpha})_{\tau_{\alpha}(i)\tau_{\alpha}(i)}^{-(\frac{1}{2}-\frac{1}{p'})x} (Q_{\alpha})_{ii}^{(\frac{1}{2}-\frac{1}{p'})x } \cdot \frac{y^{\alpha}_i(Q_{\alpha})_{ii}^{-1}}{d_{\alpha}} E^{\alpha}_{\tau_{\alpha}(i)i}.
\end{align}

An important advantage from bijectivity of $\tau_{\alpha}$ is the following identity:
\begin{align}
    &\left | \sum_{i=1}^{n_{\alpha}} (Q_{\alpha})_{\tau_{\alpha}(i)\tau_{\alpha}(i)}^{-(\frac{1}{2}-\frac{1}{p'})x} (Q_{\alpha})_{ii}^{(\frac{1}{2}-\frac{1}{p'})x } \cdot \frac{y^{\alpha}_i(Q_{\alpha})_{ii}^{-1}}{d_{\alpha}} E^{\alpha}_{\tau_{\alpha}(i)i}Q_{\alpha}^{\frac{1}{p'}} \right |\\
    &= \left ( \sum_{i=1}^{n_{\alpha}} (Q_{\alpha})_{\tau_{\alpha}(i)\tau_{\alpha}(i)}^{-(1-\frac{2}{p'})x} (Q_{\alpha})_{ii}^{(1-\frac{2}{p'})x } \frac{|y^{\alpha}_i|^2 (Q_{\alpha})_{ii}^{-2}}{d_{\alpha}^2} E^{\alpha}_{ii} \cdot (Q_{\alpha})_{ii}^{\frac{2}{p'}} \right )^{\frac{1}{2}}\\
    &= \sum_{i=1}^{n_{\alpha}} (Q_{\alpha})_{\tau_{\alpha}(i)\tau_{\alpha}(i)}^{-(\frac{1}{2}-\frac{1}{p'})x} (Q_{\alpha})_{ii}^{(\frac{1}{2}-\frac{1}{p'})x} \frac{|y^{\alpha}_i| (Q_{\alpha})_{ii}^{-1+\frac{1}{p'}}}{d_{\alpha}}E^{\alpha}_{ii}
\end{align}
Thus, it is straightforward to see that 
\begin{equation}
    \norm{\mathcal{F}^{(x)}_1(f)}_{\ell^{\infty}(\widehat{\g})}=\sup_{\alpha\in \text{Irr}(\g)}\sup_{1\leq i\leq n_{\alpha}} \left [ (Q_{\alpha})_{\tau_{\alpha}(i)\tau_{\alpha}(i)}^{-1}(Q_{\alpha})_{ii} \right ]^{\frac{x}{2}}\cdot \frac{(Q_{\alpha})_{ii}^{-1}}{d_{\alpha}}|y^{\alpha}_i|
\end{equation}
and, for any $1 < p\leq \infty$, the $p'$-norm of $\mathcal{F}^{(x)}_p(f)$ is calculated as 
\begin{align}
    \left ( \sum_{\alpha\in \text{Irr}(\g)} \sum_{i=1}^{n_{\alpha}}  \left [ (Q_{\alpha})_{\tau_{\alpha}(i)\tau_{\alpha}(i)}^{-1} (Q_{\alpha})_{ii} \right ]^{(\frac{p'}{2}-1)x} \left [\frac{(Q_{\alpha})_{ii}^{-1}}{ d_{\alpha} } \right ]^{p'-1} \left |y^{\alpha}_i \right |^{p'} \right )^{\frac{1}{p'}}.
\end{align}

\end{proof}

Applying the above Lemma \ref{lem30}, the boundedness of $\mathcal{F}^{(x)}_p$ with $p>2$ is completely answered in the following proposition.

\begin{proposition}\label{prop51}
For any $p\in (2,\infty]$, we have
\begin{equation}
I(\g,p)=\left\{\begin{array}{ll}\mathbb{R}&,~\text{if }\g\text{ is finite-dimensional}\\ \emptyset&,~\text{if }\g\text{ is infinite-dimensional} \end{array} \right . .    
\end{equation}

\end{proposition}

\begin{proof}
It is clear that $I(\g,p)=\re$ if $\g$ is finite-dimensional, so let us suppose that $\g$ is infinite-dimensional and consider the following two situations: (1) $\g$ is of Kac type and (2) $\g$ is of non-Kac type.

(1) Let us suppose that $\g$ is of Kac type and $p\geq 2$. Then the Fourier transform $f\in\text{Pol}(\g)\mapsto \widehat{f}\in c_{00}(\widehat{\g})$ extends to a bounded linear map from $L^p(\g)$ into $\ell^{p'}(\widehat{\g})$ if and only if $\g$ is finite-dimensional by \cite[Subsection 3.5]{Youn-Thesis} whose arguments rely essentially on the theory of random Fourier series.

(2) Now, let us suppose that $\g$ is of non-Kac type. Note that it is enough to take indices $i=i(\alpha)$ and $j=j(\alpha)$ such that
\begin{equation}
    \sup_{\alpha\in \text{Irr}(\g)}\frac{\norm{\mathcal{F}^{(x)}_p(u^{\alpha}_{ij})}_{\ell^{p'}(\widehat{\g})}}{\norm{u^{\alpha}_{ij}}_{L^p(\g)}}=\infty.
\end{equation}
In this case, the numerator is given by $[(Q_{\alpha})_{jj}(Q_{\alpha})_{ii}^{-1}]^{(\frac{1}{2}-\frac{1}{p})x}\cdot \left ( \frac{(Q_{\alpha})_{ii}^{-1}}{d_{\alpha}}\right )^{\frac{1}{p}}$ by Lemma \ref{lem30}, and we have
\begin{equation}
    \norm{u^{\alpha}_{ij}}_{L^p(\g)}\leq \norm{u^{\alpha}_{ij}}_{L^2(\g)}^{\frac{2}{p}}\cdot \norm{u^{\alpha}_{ij}}_{L^{\infty}(\g)}^{1-\frac{2}{p}}=  \left ( \frac{(Q_{\alpha})_{ii}^{-1}}{d_{\alpha}}\right )^{\frac{1}{p}}\cdot \norm{u^{\alpha}_{ij}}_{L^{\infty}(\g)}^{1-\frac{2}{p}}
\end{equation}
since $\frac{\frac{2}{p}}{2}+\frac{1-\frac{2}{p}}{\infty}=\frac{1}{p}+0=\frac{1}{p}$. Combining these observations, it is enough to prove 
\begin{equation}
\sup_{\alpha\in \text{Irr}(\g)}\frac{[(Q_{\alpha})_{jj}(Q_{\alpha})_{ii}^{-1}]^{(\frac{1}{2}-\frac{1}{p})x}}{\norm{u^{\alpha}_{ij}}_{L^{\infty}(\g)}^{1-\frac{2}{p}}}=\infty.
\end{equation}

Let us divide our discussions into the following two cases: (a) $x>0$ and (b) $x\leq 0$. For the case (a), let us take the indices $i=i(\alpha)$ and $j=j(\alpha)$ such that $(Q_{\alpha})_{jj}=\norm{Q_{\alpha}}_{op}$ and $(Q_{\alpha})_{ii}^{-1}=\norm{Q_{\alpha}^{-1}}_{op}$. Then we obtain
\begin{align}
    &\sup_{\alpha\in \text{Irr}(\g)}\frac{[(Q_{\alpha})_{jj}(Q_{\alpha})_{ii}^{-1}]^{(\frac{1}{2}-\frac{1}{p})x}}{\norm{u^{\alpha}_{ij}}_{L^{\infty}(\g)}^{1-\frac{2}{p}}}\geq \sup_{\alpha\in \text{Irr}(\g)}\frac{[\norm{Q_{\alpha}}_{op}\norm{Q_{\alpha}^{-1}}_{op}]^{(\frac{1}{2}-\frac{1}{p})x}}{1}=\infty
\end{align}
by Lemma \ref{lem20}. For the other case (b), let us take the indices $i=i(\alpha)$ and $j=j(\alpha)$ such that $(Q_{\alpha})_{jj}^{-1}=\norm{Q_{\alpha}^{-1}}_{op}$ and $(Q_{\alpha})_{ii}=\norm{Q_{\alpha}}_{op}$. In this case, note that we have
\begin{equation}
   \norm{u^{\alpha}_{ij}}_{L^{\infty}(\g)}=\norm{(u^{\alpha}_{ij})^*}_{L^{\infty}(\g)}\leq  (Q_{\alpha})_{ii}^{-\frac{1}{2}}(Q_{\alpha})_{jj}^{\frac{1}{2}}=[\norm{Q_{\alpha}}_{op}\norm{Q_{\alpha}^{-1}}_{op}]^{-\frac{1}{2}}
\end{equation}
since $(Q_{\alpha}^{\frac{1}{2}}\otimes 1)(u^{\alpha})^c (Q_{\alpha}^{-\frac{1}{2}}\otimes 1)$ is a unitary. Thus, we obtain
\begin{align}
    \sup_{\alpha\in \text{Irr}(\g)}\frac{[(Q_{\alpha})_{jj}^{-1}(Q_{\alpha})_{ii}]^{(\frac{1}{2}-\frac{1}{p})(-x)}}{\norm{u^{\alpha}_{ij}}_{L^{\infty}(\g)}^{1-\frac{2}{p}}}&\geq \sup_{\alpha\in \text{Irr}(\g)}\frac{[\norm{Q_{\alpha}^{-1}}_{op}\norm{Q_{\alpha}}_{op}]^{(\frac{1}{2}-\frac{1}{p})(-x)}}{[\norm{Q_{\alpha}}_{op}\norm{Q_{\alpha}^{-1}}_{op}]^{-(\frac{1}{2}-\frac{1}{p})}}\\
    =&\sup_{\alpha\in \text{Irr}(\g)}[\norm{Q_{\alpha}}_{op}\norm{Q_{\alpha}^{-1}}_{op}]^{(\frac{1}{2}-\frac{1}{p})(1-x)}=\infty
\end{align}
again by Lemma \ref{lem20}.

\end{proof}

\subsection{Boundedness for the cases $1\leq p< 2$ and $0\leq x \leq 1$}\label{subsec-3-bdd}

In this Subsection, we will focus on the cases $1\leq p< 2$. If $\g$ is of Kac type, then we have $\mathcal{F}^{(x)}_p=\mathcal{F}^{(0)}_p$ for all $x\in \re$, so we obtain $\mathbb{R}=I(\g,p)$ for all $1\leq p<2$ thanks to the standard Hausdorff-Young inequality.

However, for any non-Kac compact quantum group $\g$, the twisted Fourier transforms $\mathcal{F}^{(x)}_p$ and $\mathcal{F}^{(y)}_p$ ($x\neq y$) are not comparable to each other in the following sense.

\begin{proposition}\label{thm-incomparability}
    Let $\g$ be a non-Kac compact quantum group, and fix $p\in [1,2)$ and $x,y\in \re$. If we suppose that there exists a universal constant $K=K(\g,p,x,y)>0$ such that
    \begin{equation}
      \norm{\mathcal{F}^{(x)}_p(f)}_{\ell^{p'}(\widehat{\g})}\leq K \norm{\mathcal{F}^{(y)}_p(f)}_{\ell^{p'}(\widehat{\g})}
    \end{equation}
for all $f\in \text{Pol}(\g)$, then $x=y$ holds.
\end{proposition}

\begin{proof}
Note that Lemma \ref{lem30} implies
\begin{equation}
    \frac{\norm{\mathcal{F}^{(x)}_p(u^{\alpha}_{ij})}_{\ell^{p'}(\widehat{\g})}}{\norm{\mathcal{F}^{(y)}_p(u^{\alpha}_{ij})}_{\ell^{p'}(\widehat{\g})}}=\left [ (Q_{\alpha})_{ii} ( Q_{\alpha})_{jj}^{-1} \right ]^{(\frac{1}{p}-\frac{1}{2})(x-y)},
\end{equation}
for any matrix elements $u^{\alpha}_{ij}$. From now, let us assume $x\neq y$. 
\begin{itemize}
    \item If $x>y$, let us take indices $i=i(\alpha),j=j(\alpha)$ such that $(Q_{\alpha})_{ii}=\norm{Q_{\alpha}}_{op}$ and $(Q_{\alpha})_{jj}^{-1}=\norm{Q_{\alpha}^{-1}}_{op}$. 
    \item If $x< y$, let us take indices $i=i(\alpha),j=j(\alpha)$ such that $(Q_{\alpha}^{-1})_{ii}=\norm{Q_{\alpha}^{-1}}_{op}$ and $(Q_{\alpha})_{jj}=\norm{Q_{\alpha}}_{op}$.
\end{itemize}
Then, in both cases, we obtain
\begin{equation}
    \frac{\norm{\mathcal{F}^{(x)}_p(u^{\alpha}_{ij})}_{\ell^{p'}(\widehat{\g})}}{\norm{\mathcal{F}^{(y)}_p(u^{\alpha}_{ij})}_{\ell^{p'}(\widehat{\g})}}=\left [ \norm{Q_{\alpha}}_{op}\norm{Q_{\alpha}^{-1}}_{op} \right ]^{(\frac{1}{p}-\frac{1}{2})\left |  x-y\right |}
\end{equation}
for all $\alpha\in \text{Irr}(\g)$ by Lemma \ref{lem30}, so we can conclude that
\begin{equation}
  \sup_{f\in L^p(\g)\setminus \left\{0\right\} }  \frac{\norm{\mathcal{F}^{(x)}_p(f)}_{\ell^{p'}(\widehat{\g})}}{\norm{\mathcal{F}^{(y)}_p(f)}_{\ell^{p'}(\widehat{\g})}} \geq  \left [ \sup_{\alpha\in \text{Irr}(\g)} \norm{Q_{\alpha}}_{op}\norm{Q_{\alpha}^{-1}}_{op} \right ]^{ (\frac{1}{p}-\frac{1}{2}) \left |  x-y \right |}=\infty.
\end{equation}
The last equality is thanks to Lemma \ref{lem20}.
    
\end{proof}

The main result of this Section is that $[0,1]\subseteq I(\g,p)$ for all $1\leq p<2$. Let us begin with the case $p=1$ in the following theorem.

\begin{theorem}\label{thm-main0}
Let $\g$ be a general compact quantum group. For any $\varphi\in L^1(\g)$, we have
\begin{equation}\label{eq36}
   \sup_{0\leq x \leq 1} \norm{Q^{-\frac{x}{2}}\widehat{\varphi}Q^{\frac{x}{2}}}_{\ell^{\infty}(\widehat{\g})}\leq \norm{\varphi}_{L^1(\g)}.
\end{equation}
In particular, we have $[0,1]\subseteq I(\g,1)$.
\end{theorem}

\begin{proof}
Note that the inequality \eqref{eq36} for the case $x=0$ is trivial, so let us focus on the other extremal case $x=1$ as the first step. Recall that 
\begin{itemize}
    \item $(Q_{\alpha}^{\frac{1}{2}}\otimes 1)(u^{\alpha})^c (Q_{\alpha}^{-\frac{1}{2}}\otimes 1)$ is unitary for any $\alpha\in \text{Irr}(\g)$,
    \item $\varphi:L^{\infty}(\g)\rightarrow \Comp$ is completely bounded with $\norm{\varphi}_{cb}=\norm{\varphi}_{L^1(\g)}$.
\end{itemize} 
Thus, we have
\begin{align}
&\norm{ Q_{\alpha}^{\frac{1}{2}} (\text{id}\otimes \varphi)\left ( (u^{\alpha})^c \right ) Q_{\alpha}^{-\frac{1}{2}}}_{M_{n_{\alpha}}}\\
&=\norm{(\text{id}\otimes \varphi)\left ( (Q_{\alpha}^{\frac{1}{2}}\otimes 1)(u^{\alpha})^c (Q_{\alpha}^{-\frac{1}{2}}\otimes 1) \right )}_{M_{n_{\alpha}}}\\
&\leq \norm{\varphi}_{cb}\cdot \norm{(Q_{\alpha}^{\frac{1}{2}}\otimes 1)(u^{\alpha})^c (Q_{\alpha}^{-\frac{1}{2}}\otimes 1)}_{M_{n_{\alpha}}\otimes_{min}L^{\infty}(\g)}=\norm{\varphi}_{L^1(\g)}.
\end{align}
Furthermore, $(\text{id}\otimes \varphi)\left ( (u^{\alpha})^c \right )$ is calculated as
\begin{equation}
 \sum_{i,j=1}^{n_{\alpha}} e_{ij}\cdot  \varphi\left ((u^{\alpha}_{ij})^*\right )=\sum_{i,j=1}^{n_{\alpha}} e_{ij}\cdot \widehat{\varphi}(\alpha)_{ji}=\widehat{\varphi}(\alpha)^t,
\end{equation}
so we can conclude that
\begin{align}
\norm{Q_{\alpha}^{-\frac{1}{2}}\widehat{\varphi}(\alpha) Q_{\alpha}^{\frac{1}{2}}}_{M_{n_{\alpha}}}= \norm{Q_{\alpha}^{\frac{1}{2}}\widehat{\varphi}(\alpha)^t Q_{\alpha}^{-\frac{1}{2}}}_{M_{n_{\alpha}}}\leq \norm{\varphi}_{L^1(\g)}  
\end{align}
for all $\alpha\in \text{Irr}(\g)$.

As the second step, let us interpolate the two extremal cases $x=0$ and $x=1$ as follows. Let us fix $\alpha\in \text{Irr}(\g)$ and consider linear operators $T_z:\text{Pol}(\g)\rightarrow \ell^{\infty}(\widehat{\g})$ given by
\begin{equation}
    T_z(f)=Q^{-\frac{z}{2}}\widehat{f} Q^{\frac{z}{2}},~z\in \Comp.
\end{equation}
Then, for any real number $s$ and $f\in \text{Pol}(\g)$, we have
\begin{align}
    \norm{T_{is}(f)}_{\ell^{\infty}(\widehat{\g})}&=\norm{Q^{-\frac{is}{2}}\widehat{f}Q^{\frac{is}{2} }}_{\ell^{\infty}(\widehat{\g})}=\norm{\widehat{f}}_{\ell^{\infty}(\widehat{\g})}\leq \norm{f}_{L^1(\g)}\\
    \norm{T_{1+is}(f)}_{\ell^{\infty}(\widehat{\g})}&=\norm{Q^{-\frac{1+is}{2}}\widehat{f}Q^{\frac{1+is}{2} }}_{\ell^{\infty}(\widehat{\g})}=\norm{Q^{-\frac{1}{2}}\widehat{f}Q^{\frac{1}{2}}}_{\ell^{\infty}(\widehat{\g})}\leq \norm{f}_{L^1(\g)}
\end{align}
by the extremal cases $x=0$ and $x=1$. Applying a general Stein interpolation theorem \cite[Theorem 2.1]{Voi92}, we obtain
\begin{equation}
    \norm{T_x(f)}_{\ell^{\infty}(\widehat{\g})}=\norm{Q^{-\frac{x}{2}}\widehat{f}Q^{\frac{x}{2}}}_{\ell^{\infty}(\widehat{\g})}\leq \norm{f}_{L^1(\g)}
\end{equation}
for all $x\in [0,1]$ and $f\in L^1(\g)$.

\end{proof}

As in the classical theory, let us apply the complex interpolation theorem between the above Theorem \ref{thm-main0} and the Plancherel identity to obtain the following {\it strong} Hausdorff-Young inequality.

\begin{theorem}\label{thm-main}
Let $\g$ be a general compact quantum group. If we suppose that $x\in I(\g,1)$ with a universal constant $K=K(\g,x)$ such that 
\begin{equation}\label{eq-assumption}
\norm{Q^{-\frac{x}{2}}\widehat{\varphi}Q^{\frac{x}{2}}}_{\ell^{\infty}(\widehat{\g})}\leq K \norm{\varphi}_{L^1(\g)},~\varphi\in L^1(\g),
\end{equation}
then we have $x\in I(\g,p)$ for any $1\leq p\leq 2$ with the following inequality
\begin{equation}
\norm{Q^{-(\frac{1}{p}-\frac{1}{2})x}\widehat{f}Q^{(\frac{1}{p}-\frac{1}{2})x}}_{\ell^{p'}(\widehat{\g})} \leq K^{\frac{2}{p}-1} \norm{f}_{L^p(\g)}, f\in L^p(\g).
\end{equation}
In particular, for general compact quantum group $\g$, we have
\begin{equation}\label{eq30.5}
 \sup_{0\leq x \leq 1} \norm{Q^{-(\frac{1}{p}-\frac{1}{2})x}\widehat{f}Q^{(\frac{1}{p}-\frac{1}{2})x}}_{\ell^{p'}(\widehat{\g})} \leq  \norm{f}_{L^p(\g)}
\end{equation}
for all $f\in L^p(\g)$ with $1\leq p\leq 2$.
\end{theorem}

\begin{proof}
For the given real number $x\in I(\g,1)$, let us consider a linear map $\Phi_x:\text{Pol}(\g)\rightarrow c_{00}(\widehat{\g})$ given by
\begin{equation}
   \Phi_x(f) = Q^{-\frac{x}{2}}\widehat{f}Q^{\frac{x}{2}},~f\in \text{Pol}(\g).
\end{equation}

For each $\alpha\in \text{Irr}(\g)$, let us denote by $H_{\alpha}$ the matrix algebra $M_{n_{\alpha}}(\Comp)$ with the following norm structure
\begin{equation}\label{eq31}
    \norm{A}_{H_{\alpha}}=\norm{Q_{\alpha}^{\frac{x}{2}}A Q_{\alpha}^{\frac{1-x}{2}}}_{S^2_{n_{\alpha}}},~A\in M_{n_{\alpha}}(\Comp).
\end{equation}
Then the Plancherel theorem tells us that
\begin{itemize}
    \item [(A)] $\Phi_x: L^2(\g)\rightarrow \ell^2-(\left\{H_{\alpha}\right\}_{\alpha\in \text{Irr}(\g)},\mu)$ is an onto isometry where a positive measure $\mu$ is given by $\mu(\alpha)=d_{\alpha}$ for all $\alpha\in \text{Irr}(\g)$.
\end{itemize}
Moreover, the given assumption can be understood as that
\begin{itemize}
    \item [(B)] $\Phi_x:L^1(\g)\rightarrow \ell^{\infty}-(\left\{M_{n_{\alpha}}\right\}_{\alpha\in \text{Irr}(\g)},\mu)$ is bounded with the norm less than or equal to $K$.
\end{itemize}

Then, the complex interpolation theorem with the above (A) and (B) implies that $\Phi_x$ extends to a bounded map
\[\Phi_x:L^p(\g)\rightarrow \ell^{p'}(\left\{X_{\alpha,p'}\right\}_{\alpha\in \text{Irr}(\g)},\mu)\]
with the norm less than or equal to $K^{\frac{2}{p}-1}$, where $X_{\alpha,p'}$ is the complex interpolation space $(M_{n_{\alpha}},H_{\alpha})_{\theta}$ for $\theta=\frac{2}{p'}$. Specifically, $X_{\alpha,p'}$ is the matrix algebra $M_{n_{\alpha}}(\Comp)$ with the following norm structure
\begin{equation}
    \norm{A}_{X_{\alpha,p'}}=\norm{Q_{\alpha}^{\frac{x}{p'}}A Q_{\alpha}^{\frac{1-x}{p'}}}_{S^{p'}_{n_{\alpha}}}
\end{equation}
by \eqref{eq20.7}. Hence, we can conclude that
\begin{equation}
 \left ( \sum_{\alpha\in \text{Irr}(\g)}d_{\alpha}\text{Tr} \left ( \left |Q_{\alpha}^{\frac{x}{p'}}\cdot Q_{\alpha}^{-\frac{x}{2}}\widehat{f}(\alpha) Q_{\alpha}^{\frac{x}{2}} \cdot Q_{\alpha}^{\frac{1-x}{p'}}\right |^{p'} \right ) \right )^{\frac{1}{p'}}  \leq K^{\frac{2}{p}-1} \norm{f}_{L^p(\g)}
\end{equation}
for all $f\in L^p(\g)$. The last conclusion for \eqref{eq30.5} follows from the above discussions and Theorem \ref{thm-main0}.
\end{proof}

Recall that the standard Hausdorff-Young inequality for $\mathcal{F}^{(0)}_p$ is given by
\begin{equation}\label{eq37}
   \norm{\mathcal{F}^{(0)}_p(f)}_{\ell^{p'}(\widehat{\g})}= \norm{\widehat{f}}_{\ell^{p'}(\widehat{\g})}\leq \norm{f}_{L^p(\g)}.
\end{equation}
We will demonstrate that our strong Hausdorff-Young inequality
\begin{equation}\label{eq37.5}
\sup_{0\leq x \leq 1} \norm{\mathcal{F}^{(x)}_p(f)}_{\ell^{p'}(\widehat{\g})}= \sup_{0\leq x \leq 1} \norm{Q^{-(\frac{1}{p}-\frac{1}{2})x}\widehat{f}Q^{(\frac{1}{p}-\frac{1}{2})x}}_{\ell^{p'}(\widehat{\g})} \leq  \norm{f}_{L^p(\g)}
\end{equation}
is strictly sharper than \eqref{eq37} for some concrete examples. To do this, let us describe a situation where the left-hand side of \eqref{eq37.5} is computable.

\begin{proposition}\label{prop31}
Let $1\leq p < 2$. For each $\alpha\in \text{Irr}(\g)$, let us take indices $i=i(\alpha),j=j(\alpha)$ such that $(Q_{\alpha})_{ii}\geq (Q_{\alpha})_{jj}$. Then, for any $f\sim \sum_{\alpha\in \text{Irr}(\g)} a_{\alpha}u^{\alpha}_{ij}\in L^p(\g)$ and $x\leq y$, we have
\begin{equation}
\norm{ \mathcal{F}^{(x)}_p(f)}_{\ell^{p'}(\widehat{\g})}\leq \norm{ \mathcal{F}^{(y)}_p(f)}_{\ell^{p'}(\widehat{\g})}.
\end{equation}
In particular, we have
\begin{align}
&\sup_{0\leq x\leq 1}\norm{\mathcal{F}^{(x)}_p(f)}_{\ell^{p'}(\widehat{\g})}=\norm{\mathcal{F}^{(1)}_p(f)}_{\ell^{p'}(\widehat{\g})}\leq \norm{f}_{L^p(\g)}.
\end{align}
and the left-hand side is computed as
\begin{align}
  &\norm{\mathcal{F}^{(1)}_p(f)}_{\ell^{p'}(\widehat{\g})}=\norm{Q^{-(\frac{1}{p}-\frac{1}{2})}\widehat{f}Q^{\frac{1}{p}-\frac{1}{2}}}_{\ell^{p'}(\widehat{\g})}\\
  &= \left\{\begin{array}{lll}
  \left ( \sum_{\alpha\in \text{Irr}(\g)} |a_{\alpha}|^{p'} \left [ \frac{(Q_{\alpha})_{ii}^{-1}}{d_{\alpha}} \right ]^{p'-1} \left [ (Q_{\alpha})_{jj}^{-1}(Q_{\alpha})_{ii} \right ]^{\frac{p'}{2}-1} \right )^{\frac{1}{p'}}&,~1<p<2\\
 \sup_{\alpha\in \text{Irr}(\g)} \frac{|a_{\alpha}| }{d_{\alpha}}  [(Q_{\alpha})_{jj}(Q_{\alpha})_{ii}]^{-\frac{1}{2}} &,~p=1. \end{array} \right .
\end{align}
    
\end{proposition}

\begin{proof}

By Lemma \ref{lem30}, we have
\begin{align}
   &\norm{\mathcal{F}^{(x)}_p(f)}_{\ell^{p'}(\widehat{\g})}\\
\label{eq38}   & = \left\{\begin{array}{ll} \left ( \sum_{\alpha\in \text{Irr}(\g)}   \left [ (Q_{\alpha})_{jj}^{-1} (Q_{\alpha})_{ii} \right ]^{(\frac{p'}{2}-1)x} \left [\frac{(Q_{\alpha})_{ii}^{-1}}{ d_{\alpha} } \right ]^{p'-1} \left |a_{\alpha} \right |^{p'} \right )^{\frac{1}{p'}}&,~1<p<2\\
   \sup_{\alpha\in \text{Irr}(\g)} \left [ (Q_{\alpha})_{jj}^{-1}(Q_{\alpha})_{ii} \right ]^{\frac{x}{2}}\cdot \frac{(Q_{\alpha})_{ii}^{-1}}{d_{\alpha}}|a_{\alpha}|& p=1.
   \end{array} \right .
\end{align}
Then the given assumption $(Q_{\alpha})_{jj}^{-1}(Q_{\alpha})_{ii}\geq 1$ implies that the $p'$-norms on \eqref{eq38} increase for $x$, so the supremum
\begin{equation}
    \sup_{0\leq x\leq 1}\norm{\mathcal{F}^{(x)}_p(f)}_{\ell^{p'}(\widehat{\g})}
\end{equation}
is attained when $x=1$.

\end{proof}

\begin{example}
   Let $\g=SU_q(2)$ with $0<q<1$. Then $\text{Irr}(SU_q(2))$ is identified with $\left\{u^{(k)}\right\}_{k\in \n_0}$, and the fundamental representation $u^{(1)}$ is given by
   \begin{equation}
       u^{(1)}=\left [ \begin{array}{cc}
       a&-qc^*\\ c&a^*
       \end{array} \right ]
   \end{equation}
   where $a,c$ are the standard generators of $\text{Pol}(SU_q(2))$. Note that $u^{(k)}\in M_{k+1}(\Comp)\otimes L^{\infty}(\g)$ with $u^{(k)}_{0k}=(-qc^*)^k$. Let us consider 
    \begin{equation}
    f\sim \sum_{k=0}^{\infty} q^{-k}x_k u^{(k)}_{0k}=\sum_{k=0}^{\infty}x_k (-c^*)^k\in L^p(\g)
    \end{equation}
    whose Fourier transform is given by
    \begin{equation}
        \widehat{f}=\sum_{k=0}^{\infty}\frac{q^{-k}x_k}{d_k}E^{(k)}_{k0} Q_k^{-1}=\sum_{k=0}^{\infty}\frac{x_k}{d_k}E^{(k)}_{k0}.
    \end{equation}
    Then the standard Hausdorff-Young inequality implies
\begin{align}
\left\{\begin{array}{ll}
&\displaystyle \sup_{k\in \n_0}\frac{|x_k|}{d_k} = \norm{\mathcal{F}^{(0)}_1(f)}_{\ell^{\infty}(\widehat{\g})} \leq \norm{f}_{L^1(\g)},~p=1\\
&\displaystyle \left ( \sum_{k=0}^{\infty} \frac{q^{-k}}{d_k^{p'-1}} |x_k|^{p'}  \right )^{\frac{1}{p'}},
     = \norm{\mathcal{F}^{(0)}_p(f)}_{\ell^{p'}(\widehat{\g})}\leq \norm{f}_{L^p(\g)},~1<p<2,
     \end{array} \right .
\end{align}
whereas our strong Hausdorff-Young inequality induces a stronger inequality
\begin{align}
\label{eq30.6} \left\{\begin{array}{ll}
&\displaystyle \sup_{k\in \n_0}\frac{q^{-k}|x_k|}{d_k} = \norm{\mathcal{F}^{(1)}_1(f)}_{\ell^{\infty}(\widehat{\g})} \leq \norm{f}_{L^1(\g)},~p=1\\
&\displaystyle \left ( \sum_{k=0}^{\infty} \frac{q^{-k(p'-1)}}{d_k^{p'-1}} |x_k|^{p'}  \right )^{\frac{1}{p'}} =  \norm{\mathcal{F}^{(1)}_p(f)}_{\ell^{p'}(\widehat{\g})}\leq\norm{f}_{L^p(\g)},~1<p<2,
     \end{array} \right .
\end{align}

Additionally, we obtain
\begin{equation}\label{eq32}
 \left\{\begin{array}{ll} 
 & \displaystyle  (1-q^2) \sup_{k\in \n_0}|x_k| \leq \norm{f}_{L^1(\g)} ,~p=1,\\
 &\displaystyle  (1-q^2)^{\frac{1}{p}}  \left ( \sum_{k=0}^{\infty}  |x_k|^{p'}  \right )^{\frac{1}{p'}} \leq\norm{f}_{L^p(\g)},~1<p<2,
 \end{array} \right .
\end{equation}
thanks to \eqref{eq30.6} and the following inequality $1-q^2\leq \frac{q^{-k}}{d_k}$.

\end{example}

\section{Application to twisted rapid decay properties}\label{sec-RD}

We begin this section by explaining how the standard Hausdorff-Young inequality relates to the twisted rapid decay property introduced in \cite{BVZ}. Our approach consists of two steps. 
\begin{itemize}
    \item [{\bf [Step 1]}] The first step is to formulate the dual form of the Hausdorff-Young inequality.
    \item [{\bf [Step 2]}] The second step is to prove that this dualized inequality implies the twisted property RD under a polynomial growth on $\widehat{\g}$.
\end{itemize}

Specifically, at {\bf [Step 1]}, the dual form of the standard Hausdorff-Young inequality $\norm{\widehat{\varphi}}_{\ell^{\infty}(\widehat{\g})} \leq  \norm{\varphi}_{L^1(\g)}$ is given by
\begin{equation}\label{eq41}
    \norm{f}_{L^{\infty}(\g)}\leq \norm{\widehat{f}}_{\ell^1(\widehat{\g})}=\sum_{\alpha\in \text{Irr}(\g)} d_{\alpha}\text{Tr}\left ( \left | \widehat{f}(\alpha)Q_{\alpha} \right | \right ),~f\in \text{Pol}(\g).
\end{equation}
Then, for {\bf [Step 2]}, thanks to a basic fact $\text{Tr}(|AB|)\leq \norm{A}_{S^2_n}\norm{B}_{S^2_n}$ on \eqref{eq41}, we obtain
\begin{align}
   \norm{f}_{L^{\infty}(\g)}&\leq \sum_{\alpha\in \text{Irr}(\g)} d_{\alpha}\norm{\widehat{f}(\alpha)Q_{\alpha}}_{S^2_{n_{\alpha}}}\norm{\text{Id}_{n_{\alpha}}}_{S^2_{n_{\alpha}}}\\
  \label{eq42} &= \sum_{\alpha\in \text{Irr}(\g)} d_{\alpha} \norm{\widehat{f}(\alpha)Q_{\alpha}}_{S^2_{n_{\alpha}}}\cdot \sqrt{n_{\alpha}}\\
  &=\sum_{\alpha\in \text{Irr}(\g)}\sqrt{d_{\alpha}}\norm{\widehat{f}(\alpha)\cdot \sqrt{\frac{d_{\alpha}}{n_{\alpha}}} Q_{\alpha}^{\frac{1}{2}}\cdot Q_{\alpha}^{\frac{1}{2}}}_{S^2_{n_{\alpha}}}\cdot n_{\alpha}.
\end{align}

Let us suppose that $\g$ is a matrix quantum group, and denote by $|\cdot|$ the natural length function on $\text{Irr}(\g)$ and by
\begin{equation}
   \text{Pol}_k(\g)=\text{span}\left\{u^{\alpha}_{ij}: |\alpha|=k, 1\leq i,j\leq n_{\alpha}\right\} .
\end{equation}
Then, for any $f\in \text{Pol}_k(\g)$, it is straightforward to see that  
\begin{align}
   &\norm{f}_{L^{\infty}(\g)} \leq \sum_{\alpha\in \text{Irr}(\g): |\alpha|=k}\sqrt{d_{\alpha}}\norm{\widehat{f}(\alpha)\cdot \sqrt{\frac{d_{\alpha}}{n_{\alpha}}} Q_{\alpha}^{\frac{1}{2}}\cdot Q_{\alpha}^{\frac{1}{2}}}_{S^2_{n_{\alpha}}}\cdot n_{\alpha}\\
 \label{eq43} &\leq  \left ( \sum_{\alpha\in \text{Irr}(\g): |\alpha|=k} d_{\alpha}\cdot \norm{\widehat{f}(\alpha)\cdot \sqrt{\frac{d_{\alpha}}{n_{\alpha}}} Q_{\alpha}^{\frac{1}{2}}\cdot Q_{\alpha}^{\frac{1}{2}}}_{S^2_{n_{\alpha}}}^2  \right )^{\frac{1}{2}}\cdot \left ( \sum_{\alpha\in \text{Irr}(\g): |\alpha|=k}n_{\alpha}^2 \right )^{\frac{1}{2}}.
\end{align}
Let us write $D=(\frac{d_{\alpha}}{n_{\alpha}}Q_{\alpha})_{\alpha\in \text{Irr}(\g)}$ and assume the existence of a polynomial $P$ satisfying
\begin{equation}
    \sum_{\alpha\in \text{Irr}(\g): |\alpha|=k}n_{\alpha}^2 \leq P(k)^2,~k\in \n_0.
\end{equation}
Then the above \eqref{eq43} implies
\begin{equation}\label{eq48}
    \norm{f}_{L^{\infty}(\g)}\leq P(k)\cdot \norm{\widehat{f}D^{\frac{1}{2}}}_{\ell^2(\widehat{\g})},~f\in \text{Pol}_k(\g),
\end{equation}
hence we can recover \cite[Proposition 4.2 (a)]{BVZ}. We call \eqref{eq48} the {\it twisted property RD}.

The main goal of this Section is to establish a stronger form of the twisted property RD. While the preceding discussion relies only on the standard Hausdorff-Young inequality, we now employ our strong Hausdorff-Young inequality and revisit {\bf [Step 1]} and {\bf [Step 2]}. For {\bf [Step 1]}, we begin with the following lemma, which enables us to derive the dual formulation of Theorem \ref{thm-main}.

\begin{lemma}\label{lem40}
Let $p\in [1,\infty]$ and $x\in \re$. Then we have
\begin{equation}
   h(f^*g)=\widehat{h} \left ( \mathcal{F}^{(x)}_{p'}(f)^*\mathcal{F}^{(x)}_p(g) \right )
\end{equation}
for all $f,g\in \text{Pol}(\g)$.
\end{lemma}

\begin{proof}
First of all, we have
\begin{align}
&\mathcal{F}^{(x)}_{p'}(f)^*\mathcal{F}^{(x)}_p(g)\\
&=\left [ Q^{-(\frac{1}{p'}-\frac{1}{2})x} \widehat{f} Q^{(\frac{1}{p'}-\frac{1}{2})x} \right ]^*\cdot \left [ Q^{-(\frac{1}{p}-\frac{1}{2})x} \widehat{g} Q^{(\frac{1}{p}-\frac{1}{2})x} \right ]  \\
&= Q^{(\frac{1}{p'}-\frac{1}{2})x} \widehat{f}^* Q^{-(\frac{1}{p'}-\frac{1}{2})x}\cdot Q^{-(\frac{1}{p}-\frac{1}{2})x} \widehat{g} Q^{(\frac{1}{p}-\frac{1}{2})x} \\  
&= Q^{(\frac{1}{p'}-\frac{1}{2})x} \widehat{f}^* \widehat{g} Q^{(\frac{1}{p}-\frac{1}{2})x}
\end{align}
by definition of the twisted Fourier transform, and this implies
\begin{align}
&\text{Tr}\left ( \left [ \mathcal{F}^{(x)}_{p'}(f)^*\mathcal{F}^{(x)}_p(g) \right ](\alpha) \cdot Q_{\alpha} \right )\\
&=\text{Tr}\left ( Q_{\alpha}^{(\frac{1}{p'}-\frac{1}{2})x} \widehat{f}(\alpha)^* \widehat{g}(\alpha) Q_{\alpha}^{(\frac{1}{p}-\frac{1}{2})x}\cdot Q_{\alpha} \right )=\text{Tr}(\widehat{f}(\alpha)^*\widehat{g}(\alpha)Q_{\alpha}).
\end{align}
for each $\alpha\in \text{Irr}(\g)$. Thus, we can conclude that 
\begin{align}
   & \widehat{h}\left (  \mathcal{F}^{(x)}_{p'}(f)^*\mathcal{F}^{(x)}_p(g) \right ) =\sum_{\alpha\in \text{Irr}(\g)} d_{\alpha} \text{Tr}(\widehat{f}(\alpha)^*\widehat{g}(\alpha)Q_{\alpha}) =h(f^*g).
\end{align}
\end{proof}

Now, let us exhibit a duality relation between the twisted Fourier transforms $\mathcal{F}^{(x)}_p$ and $\mathcal{F}^{(x)}_{p'}$ in the following sense:

\begin{proposition}\label{prop-duality}
Let $p\in [1,2)$ and $x\in \re$. Then the following are equivalent.
\begin{enumerate}
\item $\norm{\mathcal{F}^{(x)}_{p}(f)}_{\ell^{p'}(\widehat{\g})}\leq K\norm{f}_{L^p(\g)}$ for all $f\in \text{Pol}(\g)$.
\item $\norm{f}_{L^{p'}(\g)}\leq K\cdot \norm{\mathcal{F}^{(x)}_{p'}(f)}_{\ell^p(\widehat{\g})}$ for all $f\in \text{Pol}(\g)$.
\end{enumerate}


\end{proposition}

\begin{proof}

Since the twisted Fourier transforms $\mathcal{F}^{(x)}_p,\mathcal{F}^{(x)}_{p'}:\text{Pol}(\g)\rightarrow c_{00}(\widehat{\g})$ are bijective, we have 
\begin{align}
 \norm{f}_{L^{p'}(\g)}&=\sup_{A\in c_{00}(\g): \norm{\left ( \mathcal{F}^{(x)}_p\right )^{-1}(A)}_{L^p(\g)}\leq 1}h(f^*\cdot \left ( \mathcal{F}^{(x)}_p\right )^{-1}(A))\\
 \label{eq44} &=\sup_{A\in c_{00}(\g): \norm{\left ( \mathcal{F}^{(x)}_{p'}\right )^{-1}(A)}_{L^p(\g)}\leq 1}h( \mathcal{F}^{(x)}_{p'}(f)^*\cdot A)
 \end{align}
by \eqref{eq21} and Lemma \ref{lem40}, and 
 \begin{align}
 \norm{A}_{\ell^{p'}(\g)}&=\sup_{g\in \text{Pol}(\g): \norm{\mathcal{F}^{(x)}_{p'}(g)}_{\ell^p(\widehat{\g})}\leq 1}\widehat{h}\left ( A^* \cdot \mathcal{F}^{(x)}_{p'}(g) \right )\\
 \label{eq45} &=\sup_{g\in \text{Pol}(\g): \norm{\mathcal{F}^{(x)}_{p'}(g)}_{\ell^p(\widehat{\g})}\leq 1}\widehat{h}\left ( \left ( \mathcal{F}^{(x)}_{p}\right )^{-1}(A)^* \cdot g \right )
\end{align}
by \eqref{eq22} and Lemma \ref{lem40}.

If we suppose that (1) holds, then we obtain \small
\begin{align}
  &\left\{ A\in c_{00}(\g): \norm{\left ( \mathcal{F}^{(x)}_p\right )^{-1}(A)}_{L^p(\g)}\leq 1\right\} \subseteq \left\{ A\in c_{00}(\widehat{\g}): \norm{A}_{\ell^{p'}(\widehat{\g})}\leq  K \right\},
\end{align}
\normalsize which implies
\begin{align}
\norm{f}_{L^{p'}(\g)}\leq \sup_{A\in c_{00}(\widehat{\g}): \norm{A}_{\ell^{p'}(\widehat{\g})}\leq  K }\widehat{h}\left ( \mathcal{F}^{(x)}_{p'}(f)^* \cdot A \right )=K\norm{\mathcal{F}^{(x)}_{p'}(f)}_{\ell^p(\widehat{\g})}
\end{align}
by \eqref{eq44}. On the other hand, the condition (2) implies
\begin{align}
  &\left\{ g\in \text{Pol}(\g): \norm{\mathcal{F}^{(x)}_{p'}(g)}_{\ell^p(\widehat{\g})}\leq 1 \right\} \subseteq \left\{ g\in \text{Pol}(\g): \norm{g}_{L^{p'}(\g)}\leq  K \right\},
\end{align}
so we obtain \small
\begin{align}
\norm{A}_{\ell^{p'}(\widehat{\g})}\leq \sup_{g\in \text{Pol}(\g): \norm{g}_{L^{p'}(\g)}\leq  K} h\left ( \left ( \mathcal{F}^{(x)}_p\right )^{-1} (A)^*\cdot g \right )=K\norm{\left (\mathcal{F}^{(x)}_{p}\right )^{-1}(A)}_{L^p(\g)}
\end{align}
\normalsize by \eqref{eq45}.

\end{proof}

Combining Theorem \ref{thm-main} with the above Proposition \ref{prop-duality}, we obtain the following dual formulation of the strong Hausdorff-Young inequality.

\begin{corollary}\label{cor-duality}
Let $1\leq p\leq 2$ and $x\in I(\g,p)$. Then there exists a universal constant $K=K(\g,p,x)$ such that
\begin{equation}
\norm{f}_{L^{p'}(\g)}\leq K \norm{\mathcal{F}^{(x)}_{p'}(f)}_{\ell^p(\widehat{\g})}
\end{equation}
for all $f\in \text{Pol}(\g)$.

\end{corollary}

Recall that the twisted property RD was defined by the following inequality 
\begin{equation}\label{eq46}
    \norm{f}_{L^{\infty}(\g)}\leq P(k)\cdot \norm{\widehat{f}D^{\frac{1}{2}}}_{\ell^2(\widehat{\g})},~f\in \text{Pol}_k(\g),
\end{equation}
with a polynomial $P$ and $D=\left (\frac{d_{\alpha}}{n_{\alpha}}Q_{\alpha} \right )_{\alpha\in \text{Irr}(\g)}$. With respect to this operator $D$, our dualized strong Hausdorff-Young inequality (Corollary \ref{cor-duality}) induces a stronger form of the twisted property RD as follows.

\begin{proposition}\label{prop40}
Let $1\leq p\leq 2$ and $x\in I(\g,p)$. Then there exists a universal constant $K=K(\g,p,x)$ such that
\begin{align}\label{eq47}
    \norm{f}_{L^{p'}(\g)}\leq K \left ( \sum_{\alpha\in \text{supp}(\widehat{f})}n_{\alpha}^2 \right )^{\frac{1}{p}-\frac{1}{2}}\cdot  \norm{D^{(\frac{1}{p}-\frac{1}{2})x}\widehat{f}D^{(\frac{1}{p}-\frac{1}{2})(1-x)}}_{\ell^2(\widehat{\g})}
\end{align}
for all $f\in \text{Pol}(\g)$.
\end{proposition}

\begin{proof}
First of all, by Corollary \ref{cor-duality}, we have 
\begin{align}
&\norm{f}_{L^{p'}(\g)}^{p}\leq K^p \sum_{\alpha\in\text{supp}(\widehat{f})}d_{\alpha}\text{Tr} \left (\left |Q_{\alpha}^{(\frac{1}{p}-\frac{1}{2})x}\widehat{f}(\alpha)Q_{\alpha}^{-(\frac{1}{p}-\frac{1}{2})x}\cdot Q_{\alpha}^{\frac{1}{p}} \right | ^p\right )\\
&=K^p \sum_{\alpha\in\text{supp}(\widehat{f})}d_{\alpha}\text{Tr} \left (\left |Q_{\alpha}^{(\frac{1}{p}-\frac{1}{2})x}\widehat{f}(\alpha)Q_{\alpha}^{(\frac{1}{p}-\frac{1}{2})(1-x)}\cdot Q_{\alpha}^{\frac{1}{2}} \right | ^p\right )
\end{align}
for all $f\in \text{Pol}(\g)$. Moreover, using the noncommutative H{\"o}lder inequality $\norm{AB}_{S^p_n}\leq \norm{A}_{S^2_n}\norm{B}_{S^y_n}$ when $\frac{1}{2}+\frac{1}{y}=\frac{1}{p}$, we obtain
\begin{align}
    &\norm{Q_{\alpha}^{(\frac{1}{p}-\frac{1}{2})x}\widehat{f}(\alpha)Q_{\alpha}^{(\frac{1}{p}-\frac{1}{2})(1-x)}\cdot Q_{\alpha}^{\frac{1}{2}}}_{S^p_{n_{\alpha}}}\\
    &\leq \norm{Q_{\alpha}^{(\frac{1}{p}-\frac{1}{2})x}\widehat{f}(\alpha)Q_{\alpha}^{(\frac{1}{p}-\frac{1}{2})(1-x)}\cdot Q_{\alpha}^{\frac{1}{2}}}_{S^2_{n_{\alpha}}}\cdot  \norm{\text{Id}_{n_{\alpha}}}_{S^y_{n_{\alpha}}}\\
    &= \norm{Q_{\alpha}^{(\frac{1}{p}-\frac{1}{2})x}\widehat{f}(\alpha)Q_{\alpha}^{(\frac{1}{p}-\frac{1}{2})(1-x)}\cdot Q_{\alpha}^{\frac{1}{2}}}_{S^2_{n_{\alpha}}}\cdot  n_{\alpha}^{\frac{1}{p}-\frac{1}{2}}\\
    &=\norm{\left [D^{(\frac{1}{p}-\frac{1}{2})x}\widehat{f}D^{(\frac{1}{p}-\frac{1}{2})(1-x)}\right ](\alpha)\cdot Q_{\alpha}^{\frac{1}{2}}}_{S^2_{n_{\alpha}}}\cdot \left ( \frac{n_{\alpha}^2}{d_{\alpha}} \right )^{\frac{1}{p}-\frac{1}{2}}.
\end{align}

Combining the above discussions, we obtain
\begin{align}
&\norm{f}_{L^{p'}(\g)}^{p}\\
&\leq K^p \sum_{\alpha\in\text{supp}(\widehat{f})}\left [d_{\alpha}^{\frac{1}{2}}\norm{\left [D^{(\frac{1}{p}-\frac{1}{2})x}\widehat{f}D^{(\frac{1}{p}-\frac{1}{2})(1-x)}\right ](\alpha)\cdot Q_{\alpha}^{\frac{1}{2}}}_{S^2_{n_{\alpha}}} \cdot n_{\alpha}^{2(\frac{1}{p}-\frac{1}{2})} \right ]^p.
\end{align}

Lastly, since $\frac{1}{2}+\frac{1}{y}=\frac{1}{p}$, applying the classical H{\"o}lder inequality once more, we can conclude that 
\begin{align}
\norm{f}_{L^{p'}(\g)}&\leq \norm{D^{(\frac{1}{p}-\frac{1}{2})x}\widehat{f}D^{(\frac{1}{p}-\frac{1}{2})(1-x)}}_{\ell^2(\widehat{\g})}\cdot \left ( \sum_{\alpha\in\text{supp}(\widehat{f})}n_{\alpha}^{2(\frac{1}{p}-\frac{1}{2})y} \right )^{\frac{1}{y}}\\
&=\norm{D^{(\frac{1}{p}-\frac{1}{2})x}\widehat{f}D^{(\frac{1}{p}-\frac{1}{2})(1-x)}}_{\ell^2(\widehat{\g})}\cdot \left ( \sum_{\alpha\in\text{supp}(\widehat{f})}n_{\alpha}^{2} \right )^{\frac{1}{p}-\frac{1}{2}}.
\end{align}

\end{proof}

Finally, we can demonstrate that our results (Theorem \ref{thm-main} and Proposition \ref{prop40}) induce a stronger and general form of the twisted property RD under the assumption of a polynomial growth.

\begin{corollary}\label{cor-strong-twisted-RD}
Let us suppose that $\g$ is a compact matrix quantum group and $|\cdot|$ is the natural length function on $\text{Irr}(\g)$. If there exists a polynomial $P$ such that 
\begin{equation}
\sum_{\alpha\in \text{Irr}(\g): |\alpha|=k}n_{\alpha}^2 \leq P(k)^2    
\end{equation}
for all $k\in \n_0$, then we have
\begin{equation}
\norm{f}_{L^{p'}(\g)} \leq P(k)^{\frac{2}{p}-1}\cdot \inf_{x\in I(\g,p)} \norm{D^{(\frac{1}{p}-\frac{1}{2})x} \widehat{f}D^{(\frac{1}{p}-\frac{1}{2})(1-x)}}_{\ell^2(\widehat{\g})}
\end{equation}
for all $f\in \text{Pol}_k(\g)$ with $1\leq p\leq 2$. In particular, for $p=1$, we have
\begin{equation}
\norm{f}_{L^{\infty}(\g)} \leq P(k)\cdot \inf_{0\leq x \leq 1} \norm{D^{\frac{x}{2}} \widehat{f}D^{\frac{1-x}{2}}}_{\ell^2(\widehat{\g})}
\end{equation}
for all $f\in \text{Pol}_k(\g)$.
\end{corollary}

\section{Characterization of boundedness of $\mathcal{F}^{(x)}_p$}\label{sec-optimality}

Recall that  Theorem \ref{thm-main} implies $[0,1]\subseteq I(\g,p)$ for all $1\leq p<2$. This Section is devoted to characterizing $I(\g,p)$ for non-Kac compact quantum groups $\g$. Subsection \ref{subsec-coamenable} focuses on non-Kac and coamenable cases, and Subsection \ref{subsec-non-coamenable} focuses on non-Kac and non-coamenable free orthogonal quantum groups $O_F^+$, respectively.

\subsection{The case of non-Kac and coamenable compact quantum groups}\label{subsec-coamenable}

The main result of this subsection is that $[0,1]=I(\g,p)$ holds for all $1\leq p<2$ under the assumption that $\g$ is non-Kac with a polynomial growth on $\widehat{\g}$. To prove this, we begin with the following lemma.
\begin{lemma}\label{lem50}
Let $1\leq p<2$, $x\in \re$, and $\alpha\in \text{Irr}(\g)$. Suppose that 
\begin{equation}
\norm{f}_{L^{p'}(\g)} \leq w( \alpha )\cdot  \norm{D^{(\frac{1}{p}-\frac{1}{2})x} \widehat{f}D^{(\frac{1}{p}-\frac{1}{2})(1-x)}}_{\ell^2(\widehat{\g})}
\end{equation}
for all $f\in \text{span}\left\{u^{\alpha}_{ij}\right\}_{1\leq i,j\leq n_{\alpha}}$. Then we have
\begin{equation}
 \left ( \frac{n_{\alpha}\norm{Q_{\alpha}^{-1}}_{op}}{d_{\alpha}} \right )^{\frac{1}{p}-\frac{1}{2}}\cdot \left [ \norm{Q_{\alpha}}_{op}\norm{Q_{\alpha}^{-1}}_{op} \right ]^{(\frac{1}{p}-\frac{1}{2})\cdot \max\left\{-x,x-1\right\} } \leq  w(\alpha).
\end{equation}

\end{lemma}

\begin{proof}

    For any matrix element $f=u^{\alpha}_{ij}$, we have
        \begin{align}
           & \norm{D^{(\frac{1}{p}-\frac{1}{2})x}\widehat{f}D^{(\frac{1}{p}-\frac{1}{2})(1-x)}}_{\ell^2(\widehat{\g})}\\
           & =\left ( d_{\alpha}\cdot \left \| \left ( \frac{d_{\alpha}}{n_{\alpha}} \right )^{\frac{1}{p}-\frac{1}{2}}Q_{\alpha}^{(\frac{1}{p}-\frac{1}{2})x}\cdot \frac{(Q_{\alpha})_{ii}^{-1}}{d_{\alpha}}E^{\alpha}_{ji} \cdot Q_{\alpha}^{(\frac{1}{p}-\frac{1}{2})(1-x)}Q_{\alpha}^{\frac{1}{2}} \right \|_{S^2_{n_{\alpha}}}^2 \right )^{\frac{1}{2}} \\
           &=\frac{\left [d_{\alpha}(Q_{\alpha})_{ii} \right ]^{\frac{1}{p}-1}}{n_{\alpha}^{\frac{1}{p}-\frac{1}{2}}}[(Q_{\alpha})_{jj}(Q_{\alpha})_{ii}^{-1}]^{(\frac{1}{p}-\frac{1}{2})x},
        \end{align}
        and $\norm{f}_{L^{p'}(\g)}\geq \norm{f}_{L^2(\g)}=\left ( \frac{(Q_{\alpha})_{ii}^{-1}}{d_{\alpha}} \right )^{\frac{1}{2}}$. Combining these calculations, our assumption implies
        \begin{align}
            w(\alpha)&\geq \left ( \frac{(Q_{\alpha})_{ii}^{-1}}{d_{\alpha}} \right )^{\frac{1}{2}}\cdot 
            \frac{n_{\alpha}^{\frac{1}{p}-\frac{1}{2}}}{\left [d_{\alpha}(Q_{\alpha})_{ii} \right ]^{\frac{1}{p}-1}} [(Q_{\alpha})_{jj}(Q_{\alpha})_{ii}^{-1}]^{-(\frac{1}{p}-\frac{1}{2})x}\\
            \label{eq50}&= \left ( \frac{n_{\alpha}(Q_{\alpha})_{ii}^{-1}}{d_{\alpha}} \right )^{\frac{1}{p}-\frac{1}{2}}[(Q_{\alpha})_{jj}(Q_{\alpha})_{ii}^{-1}]^{-(\frac{1}{p}-\frac{1}{2})x}.
        \end{align}
Here, if we take indices $i,j$ such that $(Q_{\alpha})_{ii}^{-1}=\norm{Q_{\alpha}^{-1}}_{op}$ and $(Q_{\alpha})_{jj}=\norm{Q_{\alpha}}_{op}$, then we obtain
\begin{equation}
    \left ( \frac{n_{\alpha}\norm{Q_{\alpha}^{-1}}_{op}}{d_{\alpha}} \right )^{\frac{1}{p}-\frac{1}{2}}\cdot \left [ \norm{Q_{\alpha}}_{op}\norm{Q_{\alpha}^{-1}}_{op} \right ]^{-(\frac{1}{p}-\frac{1}{2})x} \leq  w(\alpha).
\end{equation}
On the other hand, \eqref{eq50} can be written as
\begin{align}
    w(\alpha)\geq &\left ( \frac{n_{\alpha}(Q_{\alpha})_{ii}^{-1}}{d_{\alpha}} \right )^{\frac{1}{p}-\frac{1}{2}}[(Q_{\alpha})_{jj}^{-1}(Q_{\alpha})_{ii}]^{(\frac{1}{p}-\frac{1}{2})(x-1)}\cdot [(Q_{\alpha})_{jj}^{-1}(Q_{\alpha})_{ii}]^{\frac{1}{p}-\frac{1}{2}}\\
    &=\left ( \frac{n_{\alpha}(Q_{\alpha})_{jj}^{-1}}{d_{\alpha}} \right )^{\frac{1}{p}-\frac{1}{2}}[(Q_{\alpha})_{jj}^{-1}(Q_{\alpha})_{ii}]^{(\frac{1}{p}-\frac{1}{2})(x-1)}
\end{align}
Thus, if we take indices $i,j$ satisfying $(Q_{\alpha})_{jj}^{-1}=\norm{Q_{\alpha}^{-1}}_{op}$ and $(Q_{\alpha})_{ii}=\norm{Q_{\alpha}}_{op}$, then we obtain
\begin{equation}
    \left ( \frac{n_{\alpha}\norm{Q_{\alpha}^{-1}}_{op}}{d_{\alpha}} \right )^{\frac{1}{p}-\frac{1}{2}}\cdot \left [ \norm{Q_{\alpha}^{-1}}_{op}\norm{Q_{\alpha}}_{op} \right ]^{(\frac{1}{p}-\frac{1}{2})(x-1)} \leq  w(\alpha).
\end{equation}
\end{proof}

Now, let us combine some main results of Section \ref{sec-twisted} and Section \ref{sec-RD} with the above Lemma \ref{lem50} to obtain the following comprehensive theorem.

\begin{theorem}\label{thm50}
Let $\g$ be a general compact quantum group, $1\leq p<2$ and $x\in \re$. Then (1) $\Rightarrow$ (2) $\Rightarrow$ (3) $\Rightarrow$ (4) $\Rightarrow$ (5) holds in the following list.
\begin{enumerate}
    \item $0\leq x \leq 1$.
    \item $\norm{\mathcal{F}^{(x)}_p(f)}_{\ell^{p'}(\widehat{\g})}\leq \norm{f}_{L^p(\g)}$ for all $f\in \text{Pol}(\g)$.
    \item There exists a universal constant $K=K(\g,p,x)$ such that 
    \begin{equation}
        \norm{\mathcal{F}^{(x)}_p(f)}_{\ell^{p'}(\widehat{\g})}\leq K\norm{f}_{L^p(\g)}
    \end{equation}
    for all $f\in \text{Pol}(\g)$.
    \item There exists a universal constant $K=K(\g,p,x)$ such that 
    \begin{equation}
        \norm{f}_{L^{p'}(\g)}\leq K \left ( \sum_{\alpha\in \text{supp}(\widehat{f})}n_{\alpha}^2 \right )^{\frac{1}{p}-\frac{1}{2}}\cdot  \norm{D^{(\frac{1}{p}-\frac{1}{2})x}\widehat{f}D^{(\frac{1}{p}-\frac{1}{2})(1-x)}}_{\ell^2(\widehat{\g})}
    \end{equation}
    for all $f\in \text{Pol}(\g)$.
    \item There exists a universal constant $K=K(\g,p,x)$ satisfying
    \begin{align}
\label{eq52}      &\frac{n_{\alpha}\norm{Q_{\alpha}^{-1}}_{op}}{d_{\alpha}} \cdot  \left [ \norm{Q_{\alpha}^{-1}}_{op}\norm{Q_{\alpha}}_{op} \right ]^{\max\left\{-x,x-1\right\}} \leq  K \cdot n_{\alpha}^2
    \end{align}
    for all $\alpha\in \text{Irr}(\g)$.
\end{enumerate}
Additionally, if we suppose that $\g$ is a compact matrix quantum group of non-Kac type and if 
\begin{equation}\label{eq51}
    \liminf_{k\rightarrow \infty}\left ( \sum_{\alpha\in \text{Irr}(\g): |\alpha|\leq k}n_{\alpha}^2 \right )^{\frac{1}{k}}=1,
\end{equation}
then (5) $\Rightarrow$ (1) holds.
\end{theorem}

\begin{proof}
(1) $\Rightarrow$ (2) is thanks to Theorem \ref{thm-main}, (2) $\Rightarrow$ (3) is clear, (3) $\Rightarrow$ (4) $\Rightarrow$ (5) is thanks to Proposition \ref{prop40} and Lemma \ref{lem50}. Now, let us focus on (5) $\Rightarrow$ (1). By Lemma \ref{lem20}, there exists a sequence $(\alpha_n)_{n\in \n}\subseteq \text{Irr}(\g)$ such that $|\alpha_n|\leq  2^{n-1}\cdot |\alpha_1|$ and $\norm{Q_{\alpha_n}}_{op}=\norm{Q_{\alpha_1}}_{op}^{2^{n-1}}>1$ for all $n\in \n$. Then the given assumption \eqref{eq52} implies
\begin{equation}
   \norm{Q_{\alpha_1}}_{op}^{2^{n-1}\cdot \max\left\{-x,x-1\right\}}= \norm{Q_{\alpha_n}}_{op}^{\max\left\{-x,x-1\right\}}\leq K \sum_{\alpha\in \text{Irr}(\g): |\alpha|\leq 2^{n-1}|\alpha_1|}n_{\alpha}^2,
\end{equation}
so the assumption \eqref{eq51} implies
\begin{equation}
    \norm{Q_{\alpha_1}}_{op}^{\max\left\{-x,x-1\right\}}\leq \liminf_{n\rightarrow \infty} \left ( \sum_{\alpha\in \text{Irr}(\g): |\alpha|\leq 2^{n-1}\cdot |\alpha_1|}n_{\alpha}^2 \right )^{\frac{1}{2^{n-1}}}=1,
\end{equation}
which is true only when $-x\leq 0$ and $x-1\leq 0$, i.e. $x\in [0,1]$.
\end{proof}

The Drinfeld-Jimbo $q$-deformations $G_q$ are standard examples of non-Kac compact quantum groups whose duals have polynomial growth; hence, the following Corollary.

\begin{corollary}\label{cor50.5}
Let $G$ be a simply connected semisimple compact Lie group and $0<q<1$. Then $I(G_q,p)=[0,1]$ holds for any $1\leq p<2$.
\end{corollary}

\subsection{The case of non-Kac and non-coamenable free orthogonal quantum groups}\label{subsec-non-coamenable}

We turn our attention to the cases where $\g$ is non-Kac and non-coamenable. A prototypical example is the so-called free orthogonal quantum group $O_F^+$. We may assume that $F$ is of the form $\sum_{i=1}^N \lambda_i e_{i,N+1-i}\in M_N(\mathbb{R})$ satisfying that $F^2=\pm \text{Id}_N$ and $(|\lambda_i|)_{1\leq i \leq N}$ is monotone increasing. In this Section, let us suppose that $N\geq 3$ and $F$ is not a unitary. These conditions are equivalent to that $O_F^+$ is non-coamenable and non-Kac, respectively. 

For any free orthogonal quantum groups $O_F^+$, the space of irreducible unitary representations $\text{Irr}(O_F^+)$ is identified with 
\begin{equation}
    \left\{u^{(k)}\in M_{n_k}(\Comp)\otimes L^{\infty}(O_N^+): k\in \n_0\right\}
\end{equation}
and $Q_1=F^*F$, $\norm{Q_k^{-1}}_{op}=\norm{Q_k}_{op}=\norm{Q_1}_{op}^k=\norm{F}_{op}^{2k}$ for all $k\in \n_0$. Moreover, we have $n_0=1=d_0$, $n_1=N$, $d_1=\text{Tr}(F^*F)$, and
\begin{align}
\left\{\begin{array}{ll}
   & n_{k+2}=n_1n_{k+1}-n_k\\
   & d_{k+2}=d_1d_{k+1}-d_k \end{array} \right . ,~k\in \n_0.
\end{align}
This allows us to obtain the following asymptotic estimates
\begin{align}\label{eq56}
\left\{\begin{array}{ll}
   & n_k\sim \left ( \frac{n_1+\sqrt{n_1^2-4}}{2} \right )^k\\
   & d_k\sim \left ( \frac{d_1+\sqrt{d_1^2-4}}{2} \right )^k \end{array} \right . .
\end{align}
Here, $C\lesssim D $ means that there is a universal constant $K$, independent of $C,D$, satisfying
\begin{equation}
     C\leq K\cdot D,
\end{equation}
and we write $A\sim B$ if $A\lesssim B$ and $B\lesssim A$.

To demonstrate $[0,1]\subsetneq I(O_F^+,p)$, our strategy is to apply a recently developed {\it Haagerup inequality} on non-Kac free orthogonal quantum groups \cite{BVY21,You22} to establish the following lemma.

\begin{lemma}\label{lem51}
For any non-Kac free orthogonal quantum group $O_F^+$, there exists a universal constant $C=C(F)$ such that
\begin{align}
   C\cdot  \sup_{k\in \n_0}\frac{\sqrt{d_k}}{\norm{F}_{op}^{(\max\left\{|2x-1|,1\right\}+2)k}}\norm{Q_k^{-\frac{x}{2}}\widehat{f}(k)Q_k^{\frac{x}{2}}}_{op} \leq \norm{f}_{L^1(O_F^+)} 
\end{align}
for all $f\in \text{Pol}(O_F^+)$ and $x\in \re$.
\end{lemma}

\begin{proof}

Let us fix $x\in \re$ and $f\in \text{Pol}(O_F^+)$. Note that, for any $k\in \n_0$, we have
\begin{equation}\label{eq53}
    \norm{Q_{k}^{-\frac{x}{2}}\widehat{f}(k)Q_k^{\frac{x}{2}}}_{op}\leq \norm{Q_k^{-\frac{x}{2}}}_{op}\norm{\widehat{f}(k)Q_k^{\frac{1}{2}}}_{op}\norm{Q_k^{\frac{x-1}{2}}}_{op}.
\end{equation}
Let us consider the following three distinct cases $\left\{\begin{array}{lll}(1)&x>1\\ (2) &0\leq x \leq 1\\ (3) &x<0 \end{array} \right .$. For each case, the right-hand side on \eqref{eq53} is given by
\begin{equation}
    \left\{\begin{array}{ll}
    (1)& \norm{Q_k^{-\frac{1}{2}}}_{op}^x\norm{\widehat{f}(k)Q_k^{\frac{1}{2}}}_{op}\norm{Q_k^{\frac{1}{2}}}_{op}^{x-1}=\norm{F}_{op}^{(2x-1)k}\norm{\widehat{f}(k)Q_k^{\frac{1}{2}}}_{op},\\
    (2)& \norm{Q_k^{-\frac{1}{2}}}_{op}^x\norm{\widehat{f}(k)Q_k^{\frac{1}{2}}}_{op}\norm{Q_k^{-\frac{1}{2}}}_{op}^{1-x}=\norm{F}_{op}^{k}\norm{\widehat{f}(k)Q_k^{\frac{1}{2}}}_{op},\\
    (3)& \norm{Q_k^{\frac{1}{2}}}_{op}^{-x}\norm{\widehat{f}(k)Q_k^{\frac{1}{2}}}_{op}\norm{Q_k^{-\frac{1}{2}}}_{op}^{1-x}= \norm{F}_{op}^{(1-2x)k}\norm{\widehat{f}(k)Q_k^{\frac{1}{2}}}_{op}.
    \end{array} \right .
\end{equation}
To pursue simplicity, let us write $\gamma=\max\left\{|2x-1|,1\right\}$. Then the discussion above can be summarized as
\begin{align}\label{eq54}
    \norm{Q_{k}^{-\frac{x}{2}}\widehat{f}(k)Q_k^{\frac{x}{2}}}_{op}\leq \norm{F}_{op}^{\gamma k} \norm{\widehat{f}(k)Q_k^{\frac{1}{2}}}_{op}.
\end{align}

Note that, by \cite[Proposition 4.4]{You22}, there exists a universal constant $C=C(F)$ such that 
\begin{equation}\label{eq55}
   C \cdot \sup_{k\in \n_0}\frac{\sqrt{d_k}\norm{\widehat{f}(k)Q_k^{\frac{1}{2}}}_{op}}{\norm{F}_{op}^{2k}} \leq C \cdot \sup_{k\in \n_0}\frac{\sqrt{d_k}\norm{\widehat{f}(k)Q_k^{\frac{1}{2}}}_{S^2_{n_{k}}}}{\norm{F}_{op}^{2k}} \leq \norm{f}_{L^1(O_F^+)}
\end{equation}
for all $k\in \n_0$ and $f\in \text{Pol}(O_F^+)$. Thus, combining \eqref{eq54} and \eqref{eq55}, we can conclude that
\begin{align}
C\cdot  \sup_{k\in \n_0}\frac{\sqrt{d_k}}{\norm{F}_{op}^{(\gamma+2)k}}\norm{Q_k^{-\frac{x}{2}}\widehat{f}(k)Q_k^{\frac{x}{2}}}_{op} \leq \norm{f}_{L^1(\g)} .
\end{align}
\end{proof}

While Lemma \ref{lem51} will later serve as a sufficient condition for $x\in I(O_F^+,p)$, the following Lemma provides a necessary condition for $x\in I(\g,p)$.

\begin{lemma}\label{lem52}
Let $\g$ be a compact quantum group satisfying $\norm{Q_{\alpha}}_{op}=\norm{Q_{\alpha}^{-1}}_{op}$ for all $\alpha\in \text{Irr}(\g)$. If $x\in I(\g,p)$ with $1\leq p<2$, then we have $\displaystyle \sup_{\alpha\in \text{Irr}(\g)}\frac{\norm{Q_{\alpha}}_{op}^{|2x-1|}}{d_{\alpha}}<\infty$.
\end{lemma}

\begin{proof}
For any matrix element $f=u^{\alpha}_{ij}$, we have
\begin{align}
    \norm{\mathcal{F}^{(x)}_p(u^{\alpha}_{ij})}_{\ell^{p'}(\widehat{\g})}= [(Q_{\alpha})_{jj}^{-1}(Q_{\alpha})_{ii}]^{(\frac{1}{p}-\frac{1}{2})x}\cdot \left ( \frac{(Q_{\alpha})_{ii}^{-1}}{d_{\alpha}} \right )^{\frac{1}{p}}
\end{align}
by Lemma \ref{lem30}. The assumption $x\in I(\g,p)$ implies that there exists a universal constant $K=K(\g,p,x)$ such that
\begin{equation}
    [(Q_{\alpha})_{jj}^{-1}(Q_{\alpha})_{ii}]^{(\frac{1}{p}-\frac{1}{2})x}\cdot \left ( \frac{(Q_{\alpha})_{ii}^{-1}}{d_{\alpha}} \right )^{\frac{1}{p}}\leq K\cdot \norm{f}_{L^2(\g)} = K \cdot \sqrt{\frac{(Q_{\alpha})_{ii}^{-1}}{d_{\alpha}}}.
\end{equation}
This is equivalent to
\begin{equation}
 \label{eq58}   \sup_{\alpha\in \text{Irr}(\g)} \sup_{1\leq i,j\leq n_{\alpha}} \frac{ [(Q_{\alpha})_{jj}^{-1}(Q_{\alpha})_{ii}]^{x} (Q_{\alpha})_{ii}^{-1} }{d_{\alpha}}<\infty.
\end{equation}

If $x\geq \frac{1}{2}$, let us take $i,j$ such that 
\begin{equation}
    (Q_{\alpha})_{jj}^{-1}=\norm{Q_{\alpha}^{-1}}_{op}=\norm{Q_{\alpha}}_{op}=(Q_{\alpha})_{ii}.
\end{equation}
Then we obtain $\displaystyle \sup_{\alpha\in \text{Irr}(\g)} \frac{ \norm{Q_{\alpha}}_{op}^{2x-1} }{d_{\alpha}}<\infty$ by \eqref{eq58}.

If $x<\frac{1}{2}$, let us take $i,j$ such that 
\begin{equation}
    (Q_{\alpha})_{jj}=\norm{Q_{\alpha}}_{op}=\norm{Q_{\alpha}^{-1}}_{op}=(Q_{\alpha})_{ii}^{-1}.
\end{equation}
Then we obtain $\displaystyle \sup_{\alpha\in \text{Irr}(\g)} \frac{ \norm{Q_{\alpha}}_{op}^{-2x+1} }{d_{\alpha}}<\infty$ by \eqref{eq58}.
\end{proof}

Let us apply the above Lemma \ref{lem51} and Lemma \ref{lem52} to reach the following theorem.

\begin{theorem}\label{thm51}
Let $O_F^+$ be a non-Kac free orthogonal quantum group. Then  
\begin{align}
   & [0,1]\cup  \left\{x\in \re: |2x-1| \leq \frac{\log \left ( \frac{d_1+\sqrt{d_1^2-4}}{2} \right )}{2\log\norm{F}_{op}}-2\right\}\\
   &{\color{white}ttttt}\subseteq I(O_F^+,p)\subseteq  \left\{x\in \re: |2x-1| \leq \frac{\log \left ( \frac{d_1+\sqrt{d_1^2-4}}{2} \right )}{2\log\norm{F}_{op}}\right\}
\end{align}
for all $1\leq p<2$.
\end{theorem}

\begin{proof}
    For the first inclusion, it is clear that $[0,1]\subseteq I(O_F^+,p)$ by Theorem \ref{thm-main}, so let us suppose that $x\notin [0,1]$, i.e. $|2x-1|>1$. Then, by Lemma \ref{lem51}, we obtain
\begin{align}
 \label{eq59}  \sup_{k\in \n_0}\frac{\sqrt{d_k}}{\norm{F}_{op}^{(|2x-1|+2)k}}\norm{Q_k^{-\frac{x}{2}}\widehat{f}(k)Q_k^{\frac{x}{2}}}_{op}\lesssim \norm{f}_{L^1(O_F^+)} 
\end{align}
for all $f\in \text{Pol}(O_F^+)$. Let us further assume that
   \begin{equation}
         |2x-1| \leq \frac{\log \left ( \frac{d_1+\sqrt{d_1^2-4}}{2} \right )}{2\log\norm{F}_{op}}-2,
    \end{equation}
which is equivalent to 
\begin{align}
\label{eq57} \norm{F}_{op}^{|2x-1|+2} \leq \left ( \frac{d_1+\sqrt{d_1^2-4}}{2} \right )^{\frac{1}{2}}.
\end{align}
Since $\displaystyle \sqrt{d_k}\sim \left ( \frac{d_1+\sqrt{d_1^2-4}}{2} \right )^{\frac{k}{2}}$ by \eqref{eq56}, the above \eqref{eq57} implies
\begin{align}
  \norm{Q^{-\frac{x}{2}}\widehat{f}Q^{\frac{x}{2}}}_{\ell^{\infty}(\widehat{O_F^+})}\lesssim \sup_{k\in \n_0}\frac{\sqrt{d_k}}{\norm{F}_{op}^{(|2x-1|+2)k}}\norm{Q_k^{-\frac{x}{2}}\widehat{f}(k)Q_k^{\frac{x}{2}}}_{op}\lesssim \norm{f}_{L^1(O_F^+)}.
\end{align}
Thus, we can conclude that $x\in I(O_F^+,1)$, and this implies $x\in  I(O_F^+,p)$ for all $1\leq p<2$ by Theorem \ref{thm-main}.

(2) Conversely, for the second inclusion, let us suppose $x\in I(O_F^+,p)$ with $1\leq p<2$. Then, Lemma \ref{lem52} implies
\begin{equation}
    \sup_{k\in \n_0}\frac{\norm{F}_{op}^{|4x-2|k}}{d_k}<\infty,
\end{equation}
and this is equivalent to 
\begin{equation}
    \norm{F}_{op}^{|4x-2|}\leq \frac{d_1+\sqrt{d_1^2-4}}{2}
\end{equation}
since $\displaystyle d_k\sim \left ( \frac{d_1+\sqrt{d_1^2-4}}{2} \right )^{k}$ by \eqref{eq56}. Thus, we can conclude that
\begin{equation}
    |2x-1|\leq \frac{\log \left ( \frac{d_1+\sqrt{d_1^2-4}}{2} \right )}{2\log\norm{F}_{op}}.
\end{equation}
\end{proof}

The following Corollary is a direct consequence of Theorem \ref{thm51}, demonstrating that $[0,1]\subsetneq I(\g,p)$ can occur within the category of non-Kac and non-coamenable compact quantum groups.

\begin{corollary}\label{cor51}
Let $O_F^+$ be a non-Kac free orthogonal quantum group. If $\norm{F}_{op}^6<\frac{d_1+\sqrt{d_1^2-4}}{2}$, then $[0,1]\subsetneq I(O_F^+,p)$ holds for all $1\leq p<2$.
\end{corollary}

\begin{proof}
It is enough to prove that the given assumption implies
\begin{equation}
\left\{x\in \re: |2x-1|\leq \frac{\log \left ( \frac{d_1+\sqrt{d_1^2-4}}{2} \right )}{2\log\norm{F}_{op}}-2\right\}\cap [0,1]^c \neq \emptyset.
\end{equation}
Indeed, the given condition $\norm{F}_{op}^6<\frac{d_1+\sqrt{d_1^2-4}}{2}$ is equivalent to
    \begin{equation}
        3<\frac{\log \left ( \frac{d_1+\sqrt{d_1^2-4}}{2} \right )}{2\log\norm{F}_{op}},
    \end{equation}
    so we can take a real number $x$ satisfying
    \begin{equation}
       1< |2x-1|<\frac{\log \left ( \frac{d_1+\sqrt{d_1^2-4}}{2} \right )}{2\log\norm{F}_{op}}-2.
    \end{equation}
Then, the first inequality implies $x\notin [0,1]$, and the second inequality implies $x\in I(O_F^+,p)$ by Theorem \ref{thm51}.

\end{proof}

\begin{remark}
The condition $\norm{F}_{op}^6<\frac{d_1+\sqrt{d_1^2-4}}{2}$ in Corollary \ref{cor51} can be true only for non-coamenable free orthogonal quantum groups. Indeed, for non-Kac coamenable free orthogonal quantum groups, we may assume that $O_F^+$ is the twisted quantum group $SU_q(2)$ with $0<q<1$, and $F$ is canonically chosen to be $\left [\begin{array}{cc}
0& - q^{\frac{1}{2}}\\
q^{-\frac{1}{2}}&0
\end{array} \right ] $. Then, the condition $\norm{F}_{op}^6<\frac{d_1+\sqrt{d_1^2-4}}{2}$ is equivalent to
\begin{equation}
    q^{-3}<\frac{(q^{-1}+q)+(q^{-1}-q)}{2}=q^{-1},
\end{equation}
implying a contradiction $q>1$.
\end{remark}

\section{Application to a similarity problem for $L^1(\g)$}\label{sec-Similarity}

A Banach algebraic formulation of the well-known {\it similarity problem} for locally compact groups asks whether for every bounded algebra homomorphism $\Psi:L^1(G)\rightarrow B(H)$ there exists an invertible operator $T\in B(H)$ such that $T\circ \Psi(\cdot)\circ T^{-1}$ is a $*$-homomorphism, where $L^1(G)$ is the convolution algebra of $G$. This problem admits a natural extension to the setting of locally compact quantum groups \cite{BrDaSa13}. In this broader framework, however, the objects of study are completely bounded algebra homomorphisms of $L^1(\g)$, rather than bounded ones, reflecting the perspective of operator space theory. Here, the convolution product $\star$ on $L^1(\g)$ is given by
\begin{equation}
   ( \varphi\star \psi )(a)=(\varphi\otimes \psi)(\Delta(a)),~a\in L^{\infty}(\g).
\end{equation}

Motivated by this perspective, the construction of explicit examples of bounded algebra homomorphisms of $L^1(\g)$ that fail to be completely bounded has emerged as an important problem. To the best of our knowledge, the only known examples arise from the duals of non-abelian free groups \cite{ChSa13} and reduced free products of compact quantum groups \cite{BrDaSa13}. Aside from the reduced free product technique, no examples of bounded but non-completely bounded algebra homomorphisms of $L^1(\g)$ are currently available.

In order to address this gap for non-Kac compact quantum groups, we introduce a new family of twisted Fourier transforms. Let us begin with the following completely contractive map $\pi^{(0)}:L^1(\g)\rightarrow \ell^{\infty}(\widehat{\g})$ defined by
\begin{equation}
   \left [ \pi^{(0)}(\varphi) \right ] (\alpha)= (\text{id}\otimes \varphi )\left ( u^{\alpha}\right ) = \left (\varphi(u^{\alpha}_{ij})\right )_{1\leq i,j\leq n_{\alpha}}
\end{equation}
for all $\alpha\in \text{Irr}(\g)$. A strong connection between $\pi^{(0)}$ and our standard Fourier transform $\mathcal{F}^{(0)}_1$ is described by
\begin{equation}\label{eq500}
\mathcal{F}^{(0)}_1(\varphi)=\pi^{(0)}(\varphi\circ S),~\varphi\in L^1(\g),
\end{equation}
where $S$ is the {\it antipode map}. In particular, if $\g$ is a compact group $G$, then \eqref{eq500} reduces to
\begin{equation}
    \mathcal{F}^{(0)}_1(f)=\pi^{(0)}(\check{f}),~f\in L^1(G),
\end{equation}
where $\check{f}(x)=f(x^{-1})$ for all $x\in G$.

While $\pi^{(0)}$ is slightly different from our standard Fourier transform $\mathcal{F}^{(0)}_1$, an advantage of this approach is an application to the study of the similarity property of the convolution algebra $L^1(\g)$. For any $x\in \re$ and $\varphi\in L^1(\g)$, let us consider
\begin{equation}
    \pi^{(x)}(\varphi)=Q^{\frac{x}{2}}\pi^{(0)}(\varphi)Q^{-\frac{x}{2}}\in \prod_{\alpha\in \text{Irr}(\g)}M_{n_{\alpha}}(\Comp)
\end{equation}
where $Q=(Q_{\alpha})_{\alpha\in \text{Irr}(\g)}$.

\begin{proposition}\label{prop500}
For any $0\leq x\leq 1$, the map $\pi^{(x)}:L^1(\g)\rightarrow \ell^{\infty}(\widehat{\g})$ is a well-defined contractive algebra homomorphism of $L^1(\g)$ into $\ell^{\infty}(\widehat{\g})$. 
\end{proposition}

\begin{proof}
First, let us focus on the extremal case $x=1$. Then, as in the proof of Theorem \ref{thm-main0}, note that $(Q_{\alpha}^{\frac{1}{2}}\otimes 1)(u^{\alpha})^c (Q_{\alpha}^{-\frac{1}{2}}\otimes 1)$ is a unitary for any $\alpha\in \text{Irr}(\g)$. Thus,
\begin{align}
    \left [(Q_{\alpha}^{\frac{1}{2}}\otimes 1)(u^{\alpha})^c (Q_{\alpha}^{-\frac{1}{2}}\otimes 1) \right ]^*=\sum_{i,j=1}^{n_{\alpha}}Q_{\alpha}^{-\frac{1}{2}}e^{\alpha}_{ji}Q_{\alpha}^{\frac{1}{2}}\otimes u^{\alpha}_{ij}
\end{align}
is also a unitary element in $M_{n_{\alpha}}(\Comp)\otimes L^{\infty}(\g)$. This implies 
\begin{align}
\norm{ (\text{id}\otimes \varphi)\left ( \sum_{i,j=1}^{n_{\alpha}}Q_{\alpha}^{-\frac{1}{2}}e^{\alpha}_{ji}Q_{\alpha}^{\frac{1}{2}}\otimes u^{\alpha}_{ij} \right )}_{M_{n_{\alpha}}}\leq \norm{\varphi}_{L^1(\g)}.
\end{align}
Since the left-hand side is given by
\begin{equation}
    \sum_{i,j=1}^{n_{\alpha}}Q_{\alpha}^{-\frac{1}{2}}e^{\alpha}_{ji}Q_{\alpha}^{\frac{1}{2}}\cdot \varphi(u^{\alpha}_{ij})=Q_{\alpha}^{-\frac{1}{2}}[\pi^{(0)}(\varphi)](\alpha)^tQ_{\alpha}^{\frac{1}{2}}.
\end{equation}
and the transpose map is contractive on $M_{n_{\alpha}}(\Comp)$, we obtain
\begin{align}
 \norm{\pi^{(1)}(\varphi)}_{\ell^{\infty}(\widehat{\g})}=\sup_{\alpha\in \text{Irr}(\g)} \norm{Q_{\alpha}^{\frac{1}{2}}[\pi^{(0)}(\varphi)](\alpha) Q_{\alpha}^{-\frac{1}{2}}}_{M_{n_{\alpha}}}\leq \norm{\varphi}_{L^1(\g)}.
\end{align}
Then, the contractivity of $\pi^{(x)}$ for arbitrary $0\leq x \leq 1 $ follows from the complex interpolation method as in the second step of the proof of Theorem \ref{thm-main0}.

Second, for any $\varphi,\psi\in L^1(\g)$ and $\alpha\in \text{Irr}(\g)$,
\begin{align}
   \left [ \pi^{(0)}(\varphi\star \psi) \right ](\alpha)= (\text{id}\otimes ( \varphi\star \psi) )(u^{\alpha}).
\end{align}
If we write $u^{\alpha}=\sum_{i,j=1}^{n_{\alpha}}e^{\alpha}_{ij}\otimes u^{\alpha}_{ij}$, then the right-hand side is given by
\begin{align}
&(\text{id}\otimes \varphi \otimes \psi)(\text{id}\otimes \Delta)(u^{\alpha})=\sum_{i,j=1}^{n_{\alpha}}\sum_{k=1}^{n_{\alpha}}e^{\alpha}_{ij}\cdot \varphi( u^{\alpha}_{ik})\cdot \psi( u^{\alpha}_{kj})\\
&=\sum_{i,j=1}^{n_{\alpha}}\sum_{k=1}^{n_{\alpha}} e^{\alpha}_{ij}\cdot \left [\pi^{(0)}(\varphi)\right ](\alpha)_{ik} \left [\pi^{(0)}(\psi)\right ](\alpha)_{kj}\\
&=\left [\pi^{(0)}(\varphi)\right ](\alpha) \left [\pi^{(0)}(\psi)\right ](\alpha)
\end{align}
for all $\alpha\in \text{Irr}(\g)$. Thus, $\pi^{(0)}$ is an algebra homomorphism from $L^1(\g)$ into $\ell^{\infty}(\widehat{\g})$. Moreover,
\begin{align}
    &\pi^{(x)}(\varphi \star \psi) = Q^{\frac{x}{2}}\pi^{(0)}(\varphi\star \psi)Q^{-\frac{x}{2}}=Q^{\frac{x}{2}}\pi^{(0)}(\varphi)\pi^{(0)}(\psi)Q^{-\frac{x}{2}}\\
    &=Q^{\frac{x}{2}}\pi^{(0)}(\varphi)Q^{-\frac{x}{2}}\cdot Q^{\frac{x}{2}}\pi^{(0)}( \psi)Q^{-\frac{x}{2}}=\pi^{(x)}(\varphi)\pi^{(x)}(\psi),
\end{align}
thus $\pi^{(x)}$ is an algebra homomorphism from $L^1(\g)$ into $\ell^{\infty}(\widehat{\g})$ for all $0\leq x \leq 1$.
\end{proof}

Let us focus on the special case $x=1$. Then the contractive algebra homomorphism $\pi^{(1)}$ is not completely bounded in general by the following theorem.

\begin{theorem}\label{thm500}
    Let $\g$ be a non-Kac compact quantum group satisfying 
    \begin{equation}\label{eq501}
    \sup_{\alpha\in \text{Irr}(\g)}\frac{n_{\alpha}\max\left\{\norm{Q_{\alpha}},\norm{Q_{\alpha}^{-1}}\right\}}{d_{\alpha}}=\infty    
    \end{equation}
    Then the contractive algebra homomorphism $\pi^{(1)}$ is not completely bounded.
\end{theorem}

\begin{proof}
By \cite[Theorem 2.5.2]{Pi03}, if we assume that $\pi^{(1)}$ is completely bounded, then there exists $V\in \ell^{\infty}(\widehat{\g})\overline{\otimes}L^{\infty}(\g)$ such that
\begin{equation}
    ( \text{id}\otimes \varphi)(V)=\pi^{(1)}(\varphi)= Q^{\frac{1}{2}} \pi^{(0)}(\varphi) Q^{-\frac{1}{2}} 
\end{equation}
for all $\varphi\in L^1(\g)$. This implies that, if we write $V=(V(\alpha))_{\alpha\in \text{Irr}(\g)}$ with $V(\alpha)\in M_{n_{\alpha}}(\Comp)\otimes L^{\infty}(\g)$, then we should have
\begin{align}
    V(\alpha)=(Q_{\alpha}^{\frac{1}{2}}\otimes 1)u^{\alpha}(Q_{\alpha}^{-\frac{1}{2}}\otimes 1).
\end{align}
Then, it is straightforward to see that
\begin{align}
    (\text{id}\otimes h)(V(\alpha)V(\alpha)^*)&=\frac{n_{\alpha}}{d_{\alpha}}Q_{\alpha}^{-1},\\
    (\text{id}\otimes h)(V(\alpha)^*V(\alpha))&=\frac{n_{\alpha}}{d_{\alpha}}Q_{\alpha},
\end{align}
and this implies $\norm{V(\alpha)}^2\geq \frac{n_{\alpha}\max\left\{\norm{Q_{\alpha}},\norm{Q_{\alpha}^{-1}}\right\}}{d_{\alpha}}$. Hence, the given assumption \eqref{eq501} implies the following contradiction
\begin{align}
\norm{V}_{\ell^{\infty}(\widehat{\g})\otimes L^{\infty}(\g)}^2 =  \sup_{\alpha\in \text{Irr}(\g)}\norm{V(\alpha)}_{M_{n_{\alpha}}\otimes L^{\infty}(\g)}^2=\infty.
\end{align}

\end{proof}

\begin{example}
In particular, $\pi^{(1)}$ is not completely bounded for all non-Kac free orthogonal quantum groups $O_F^+$ since the sufficient condition \eqref{eq501} is satisfied. Here, $F\in GL_N(\Comp)$ such that $F\overline{F}=\pm \text{Id}_N$, $\text{Irr}(O_F^+)\cong \n_0=\left\{0,1,2,\cdots\right\}$ and $\norm{Q_k}=\norm{Q_k^{-1}}$ holds for all $k\in \n_0$. If $N=2$, then $L=\inf_{k\in \n_0}\frac{\norm{Q_k}}{d_k}>0$, so it is straightforward to see
\begin{equation}
   \sup_{k\in \n_0} \frac{n_k\norm{Q_k}}{d_k}\geq  \left (\sup_{k\in \n_0} n_k\right ) \cdot L=\infty.
\end{equation}
If $N\geq 3$, an exponential growth of $\frac{n_k\norm{Q_k}}{d_k}$ is proved in the proof of \cite[Theorem 3.3]{BVY21}, so this implies
\begin{equation}
   \sup_{k\in \n_0} \frac{n_k\norm{Q_k}}{d_k}=\infty.
\end{equation}
\end{example}


\bibliographystyle{alpha}
\bibliography{Youn23}

\end{document}